\documentclass[journal,twoside,web]{ieeecolor}
\usepackage{generic}
\usepackage{cite}
\usepackage{amsmath,amssymb,amsfonts}
\usepackage{booktabs}
\usepackage{tabularx}
\usepackage{graphicx}
\usepackage{algorithm}
\usepackage{textcomp}
\usepackage[utf8]{inputenc} 
\usepackage[english]{babel} 
\graphicspath{{images/}} 
\usepackage{mathrsfs}
\usepackage{Macros}
\usepackage[usenames, dvipsnames]{xcolor}

\usepackage{todonotes}
\usepackage{placeins}
\usepackage{algpseudocode}

\usepackage{multicol}
\usepackage{braket}

\newtheorem{lemma}{Lemma}
\newtheorem{proposition}[lemma]{Proposition}
\newtheorem{theorem}[lemma]{Theorem}
\newtheorem{assumption}[lemma]{Assumption}

\newtheorem{remark}[lemma]{Remark}
\newtheorem{definition}[lemma]{Definition}

\def\proofof[#1]{\noindent\hspace{2em}{\itshape #1: }}

\usepackage{tikz}
\usepackage{pgfplots}
\definecolor{customblue}{RGB}{6,61,121}
\pgfplotsset{compat=1.17}
\pgfplotsset{
standard/.style={
			set layers=tick labels on top
		},
layers/tick labels on top/.define layer set=
		{axis background,axis grid,axis ticks,axis lines,main,%
			axis tick labels,
			axis descriptions,axis foreground}
		{/pgfplots/layers/standard}
  }
\usetikzlibrary{cd}
\usetikzlibrary{shapes,fit} 

\usetikzlibrary{decorations.markings}
\tikzset{degil/.style={
		decoration={markings,
			mark= at position 0.5 with {
				\node[transform shape] (tempnode) {$\backslash$};
			}
		},
		postaction={decorate}
	}
}
\allowdisplaybreaks[4]  
\usepackage{ifdraft}
\ifdraft{
    \usepackage{todonotes}      
	\usepackage[normalem]{ulem}
    \usepackage[displaymath, mathlines,switch]{lineno}         
    
}{}

\usepackage[hyphens]{url}					
\usepackage[hidelinks]{hyperref}			
\hypersetup{final=true}

\usepackage{soul} 

\begin{document}
\title{A novel switched systems approach to nonconvex optimisation}
\author{Joel Ferguson, Saeed Ahmed, Juan E. Machado, Michele Cucuzzella, and Jacquelien M. A. Scherpen
\thanks{Submitted on September 6, 2024. This work was supported by the Dutch
Research Council (NWO) under Grant ESI.2019.005 and the German Federal Government, the Federal Ministry of Education and Research, and the State of Brandenburg within the framework of the joint project EIZ: Energy Innovation Center (project numbers 85056897 and 03SF0693A).}
\thanks{Joel Ferguson is with the Maynooth International Engineering College, Maynooth University, Ireland (email: joel.ferguson@mu.ie).}
\thanks{Saeed Ahmed, Michele Cucuzzella, and Jacquelien M. A. Scherpen are with the Jan C. Willems Center for Systems
and Control and the Engineering, and Technology Institute Groningen (ENTEG), Faculty of Science and Engineering, University of Groningen, 9747 AG Groningen, The Netherlands (e-mails: s.ahmed@rug.nl, m.cucuzzella@rug.nl, j.m.a.scherpen@rug.nl).  }
\thanks{Juan E. Machado is with the Brandenburg University of Technology Cottbus-Senftenberg BTU, 03046 Cottbus, Germany (e-mail: machadom@b-tu.de).}
}

\maketitle

\begin{abstract}
We develop a novel switching dynamics that converges to the Karush–Kuhn–Tucker (KKT) point of a nonlinear optimisation problem. This new approach is particularly notable for its lower dimensionality compared to conventional primal-dual dynamics, as it focuses exclusively on estimating the primal variable. Our method is successfully illustrated on general quadratic optimisation problems, the minimisation of the classical Rosenbrock function, and a nonconvex optimisation problem stemming from the control of energy-efficient buildings.
\end{abstract}


\section{Introduction}
{
Solving optimisation problems is at the core of many engineering systems, ranging from airline scheduling to financial markets and machine learning \cite{nocedalNumericalOptimization1999}. Optimisation problems are characterised by a cost function to be minimised (or maximised), often subject to equality and/or inequality constraints that any feasible solution must satisfy \cite{boydConvexOptimization2004}. Problems of this form are commonly categorised as either being convex or nonconvex, with convex problems allowing much stronger convergence guarantees of the corresponding optimisation algorithms. Nonconvex problems, on the other hand, are significantly more challenging to solve due to the existence of multiple solutions and stationary points.
}

{
One of the main challenges in solving constrained optimisation problems is handling inequality constraints, which may either be active (binding) at a solution or have no effect. This complicates the search for a solution, as it is not known a priori which constraints will be active. To resolve this issue, optimisation routines define a series of sub-problems that eliminate direct handling of inequality constraints but ensure they are satisfied in the limit. Active set methods treat a subset of constraints as binding to define a sub-problem subject only to equality constraints, updating the active set iteratively until a solution is found \cite{nocedalNumericalOptimization1999}. Interior point methods, by contrast, modify the objective function with barrier functions that ensure intermediate values remain within the feasible set. The barrier function perturbs the solution, but this is resolved by defining a sequence of sub-problems, where the solutions converge to the original optimisation problem's solution \cite{boydConvexOptimization2004}.
}

{
A constrained optimisation problem can be described by a single scalar function called the Lagrangian, which combines the objective function with the constraints, weighted by Lagrange multipliers. The Lagrangian defines a saddle function for the optimisation problem, where it is minimised with respect to the decision variables (primal variables) and maximised with respect to the Lagrange multipliers (dual variables). This construction has led to a class of optimisation solutions called primal-dual methods that simultaneously attempt to solve for the primal and dual variables, rather than solving only for the primal variables. Primal-dual methods can be implemented as an active set method \cite{forsgrenPrimalDualActiveset2016} or an interior-point method \cite{boydConvexOptimization2004}.
}


{
While primal-dual optimisation methods were originally developed in discrete time, continuous-time primal-dual dynamics have been widely used as an analytical and modeling framework for distributed optimisation problems. In particular, the work of \cite{kellyRateControlCommunication1998} introduced a continuous-time model for internet congestion control, in which primal-dual dynamics arise as an approximation of distributed rate-control mechanisms. In this setting, the optimisation algorithm is represented as an ordinary differential equation defined by the gradient of the Lagrangian, enabling the study of stability and convergence of the resulting closed-loop system. Adopting a continuous-time formulation allows stability of the system to be analysed independently of discretisation effects and enables the use of tools from nonlinear systems theory \cite{maDistributedApproachesSolving2016}. This continuous time perspective is particularly natural for time-varying optimisation problems \cite{bianchinTimeVaryingOptimizationLTI2022}, distributed optimisation over networks \cite{wangSolvingSystemsLinear2014}, and settings where optimisation dynamics are interconnected with physical control systems for online optimisation \cite{steginkUnifyingEnergyBasedApproach2017}. Consequently, continuous-time primal-dual dynamics have found widespread application in communication networks \cite{kellyRateControlCommunication1998}, resource allocation \cite{feijerStabilityPrimalDual2010}, distributed convex optimisation \cite{wangSolvingSystemsLinear2014}, and optimisation over power networks \cite{dallaneseDistributedOptimalPower2013,dorflerBreakingHierarchyDistributed2016,steginkUnifyingEnergyBasedApproach2017}.
}

{
The convergence properties of primal-dual dynamics have been widely studied, particularly for strictly convex problems. In this setting, it has been shown that the solution of the optimisation problem is asymptotically stable \cite{cherukuriAsymptoticConvergenceConstrained2016}, \cite{cherukuriDistributedCoordinationPower2017}, \cite{cherukuriRoleConvexitySaddlePoint2018}, \cite{nguyenContractionRobustnessContinuous2018}, \cite{yamashitaPassivitybasedGeneralizationPrimal2020}. Further analysis has demonstrated that, when linear constraints are present, the method exhibits exponential stability \cite{quExponentialStabilityPrimalDual2019}, \cite{ozaslanGlobalExponentialStability2023}, \cite{ozaslanTightLowerBounds2023}. Beyond the strictly convex case, extensions to non-convex problems have been considered. For quasi-convex problems \cite{cherukuriSaddlePointDynamicsConditions2017} and problems with quadratic constraints \cite{maDistributedApproachesSolving2016}, the stability analysis confirms only local asymptotic stability under specific technical conditions. An alternative approach to handling non-strictly convex problems involves convexifying the Lagrangian by adding a regularisation term, constructing an augmented Lagrangian that ensures strict convexity. This method, explored in \cite{richertRobustDistributedLinear2015}, \cite{cherukuriDistributedCoordinationPower2017}, \cite{steginkUnifyingEnergyBasedApproach2017}, relaxes the strict convexity requirement but still assumes convexity of the underlying objective functions. When problems involve inequality constraints, an additional consideration is ensuring the positivity of Lagrange multipliers. To enforce this, projection operators are applied to the dual dynamics, extending primal-dual dynamics into the non-smooth setting, which includes switched systems \cite{cherukuriAsymptoticConvergenceConstrained2016}, \cite{cherukuriRoleConvexitySaddlePoint2018}, \cite{yamashitaPassivitybasedGeneralizationPrimal2020}. This extension ensures that the Lagrange multipliers remain within the feasible region, maintaining the validity of the optimisation process.
}

In this paper, we propose a novel algorithm for solving a class of nonconvex optimisation problems in continuous time. The algorithm introduces switching dynamics which, under a technical assumption on the constraints and objective function, provably converge to a KKT point. The switching law determines which inequality constraints are binding, analogous to active set methods, and ensures invariance of the feasible set, similar to interior-point approaches. In contrast to standard primal-dual dynamics that evolve both primal and dual variables, the proposed method computes only the primal variables, resulting in a lower-dimensional optimisation dynamics. A further consequence of this formulation is that damping can be included in all coordinates, allowing the objective function to serve as an ISS-Lyapunov function for stability analysis. The effectiveness of the proposed algorithm is demonstrated through case studies including quadratic programming, constrained minimisation of Rosenbrock’s function, and optimisation of steady-state operating points in HVAC systems.

\emph{Notation}: 
The notation is simplified wherever no confusion can arise from the context. We denote a null matrix of size $n\times m$ by 
$0_{n\times m}$, and an identity matrix of size $n\times n$ by  $I_n$.
For a function $f:\mathbb{R}^n \to \mathbb{R}$, we denote the transposed gradient as $\nabla f:= \left(\frac{\partial f}{\partial x}\right)^\top$.
The $i^{th}$ standard basis vector is denoted by  $e_i$.
The set difference is denoted by  $A\backslash B$.
For a function $g:\mathbb{R}\to G$, where $G$ is the image of the function $g$, $g(t^+) = \lim_{\tau\to t^+}\sigma(\tau)$ is the limit from above.

\section{Background and Problem Statement}\label{sec:probStatement}
In this paper, we consider the nonlinear optimisation problem:
\begin{equation}\label{probFormulation}
	\begin{split}
\underset{z\in\mathbb{R}^n}{\mathrm{min}} \ f(z) \ \
		&\text{s.t.} \ \underbrace{A_{eq}(z)\nabla_z f + d_{eq}}_{:= g_{eq}(z)} = 0_{m\times 1} \\
		&\phantom{s.t. }\ \underbrace{A_{ineq}(z)\nabla_z f + d_{ineq}}_{:= g_{ineq}(z)} \leq 0_{p\times 1},
	\end{split}
\end{equation}
where $z\in\mathbb{R}^n$ is a vector of $n$ decision variables, and $f:\mathbb{R}^n\rightarrow\mathbb{R}$ is the cost function that is assumed to be differentiable and radially unbounded with a radially unbounded Jacobian, i.e.,  $f(z), ||\nabla_z f(z)||\to\infty$ as $||z||\to\infty$. The problem is subject to $m$ nonlinear equality constraints $g_{eq}: \mathbb{R}^n\rightarrow\mathbb{R}^m$ described by a full rank matrix $A_{eq}(z)\in\mathbb{R}^{m\times n}$ and a constant vector $d_{eq}\in\mathbb{R}^m$. Additionally, the problem must satisfy $p$ nonlinear inequality constraints $g_{ineq}:\mathbb{R}^n\rightarrow \mathbb{R}^p$ described by a full rank matrix $A_{ineq}(z)\in\mathbb{R}^{p\times n}$ and a constant vector $d_{ineq}\in\mathbb{R}^p$.

To ensure that there exists a solution to \eqref{probFormulation}, the feasibility set is assumed to be non-empty. This is formalised by first defining the closed set on which the inequality constraints are satisfied as
\begin{equation*}\label{domainDefinition}
	\mathcal{G}_{ineq}
	=
	\left\lbrace
		z\in\mathbb{R}^n \ | \ g_{ineq}(z) \leq 0_{p\times 1}
	\right\rbrace.
\end{equation*}
If there are no inequality constraints ($p = 0$), then $\mathcal{G}_{ineq} = \mathbb{R}^n$. We additionally label the set on which the equality constraints are satisfied by
\begin{equation*}\label{GsetDefinition}
	\begin{split}
		\mathcal{G}_{eq} 
		&= 
		\big\lbrace
			z\in\mathbb{R}^n \ | \ g_{eq}(z) = 0_{m\times 1}
		\big\rbrace.
	\end{split}
\end{equation*}
Moving forward, the feasibility set $\Gamma$ is defined as the intersection of the sets $\mathcal{G}_{ineq}$ and $\mathcal{G}_{eq}$, i.e.,
\begin{equation}
    \label{feasability_set}
\Gamma:=\mathcal{G}_{ineq}\cap\mathcal{G}_{eq},
\end{equation}
which is assumed to be non-empty.

\begin{remark}
	Note that in contrast to standard formulations of optimisation problems, the constraints in \eqref{probFormulation} are defined as a function of the gradient of the objective function. {This formulation was selected primarily for technical reasons, but captures several problems of practical interest. As the matrices $A_{eq}(z), A_{ineq}(z)$ are nonlinear functions of $z$,  it is shown in Section~\ref{sec:examples} that quadratic programming problems and certain nonconvex optimisation problems, motivated from practice, can be recast into the form of \eqref{probFormulation}.}
\end{remark}

\begin{remark}
	The requirement of $||\nabla_z f(z)||$ being radially unbounded is used to ensure the boundedness of the trajectories of the proposed optimisation dynamics. {While not common within optimisation applications, this condition is often used when analysing ISS systems \cite{satohInputtostateStabilityStochastic2017}.} This requirement is satisfied, for example, by positive-definite quadratic functions.
\end{remark}

\subsection{First-order optimality conditions}
\label{sec:probStatement:KKT}
The Lagrangian function corresponding to the optimisation problem~\eqref{probFormulation} can be defined as
\begin{equation}\label{Lagrangian}
	\mathcal{L}(z,\lambda,\nu)
	=
	f(z) + \lambda^\top g_{eq}(z) + \nu^\top g_{ineq}(z),
\end{equation}
where the vectors $\lambda\in\mathbb{R}^m$ and $\nu\in\mathbb{R}^p$ are Lagrange multipliers associated with the equality and inequality constraints, respectively. Using the notion of Lagrange multipliers, the first-order necessary conditions for optimality, also referred to as the KKT conditions, can be stated as follows: 

\begin{theorem}[Theorem 12.1 \cite{nocedalNumericalOptimization1999}]
	Suppose that the cost function $f$ and the constraints $g_{eq}$ and $g_{ineq}$ are differentiable at a point $z^\star\in\mathbb{R}^n$. {Moreover, assume that the gradients of the constraints $g_{eq,i}(z^\star)$, $g_{ineq,j}(z^\star)$ are all linearly independent for $i\in\left\lbrace1,\dots,m\right\rbrace$ and $j\in\left\lbrace1,\dots,p\right\rbrace$ such that $g_{ineq,j}(z^\star) = 0$.} If $z^\star$ is a locally optimal solution of the problem \eqref{probFormulation}, then there exist constant vectors $\lambda^\star\in\mathbb{R}^m$ and $\nu^\star\in\mathbb{R}^p$ such that 
	\begin{subequations}\label{KKTconditions}
		\begin{align}
			&\nabla f(z^\star)
			+ \frac{\partial^\top g_{eq}}{\partial z}(z^\star)\lambda^\star
			+ \frac{\partial^\top g_{ineq}}{\partial z}(z^\star) \nu^\star
			= 0_{n\times 1} \label{KKTconditions:gradient} \\
			& g_{eq}(z^\star)=0_{m\times 1}\\
			& g_{ineq}(z^\star)\leq 0_{p\times 1} \label{KKTconditions:gineqNegative}\\
			&\nu^\star \geq 0_{p\times 1} \label{KKTconditions:nuPositive}\\
			&\nu^\star \circ g_{ineq}(z^\star) = 0_{p\times 1}, \label{KKTconditions:complimentary}
		\end{align}
	\end{subequations}
	where $\circ$ denotes the Hadamard product.
\end{theorem}

\subsection{Second-order sufficient condition}\label{sec:probStatement:sufficientCond}
Note that the KKT conditions~\eqref{KKTconditions}, when applied to convex optimisation problems, are both necessary and sufficient as the duality gap between the primal and dual problems is zero \cite{boydConvexOptimization2004}. In this work, however, we consider optimisation problems that are not necessarily convex, which implies that the KKT conditions \eqref{KKTconditions} are only necessary for local optimality. For nonconvex optimisation problems, the search space may contain saddle points and local maxima that would satisfy the first-order conditions without being locally optimal. Once a point satisfying the KKT conditions is found, additional examination should be performed to determine if a candidate point is indeed locally optimal. To this end, a second-order sufficient condition for local optimality can be introduced by considering the second partial derivative of the Lagrangian function \eqref{Lagrangian} with respect to $z$. To introduce such a condition, consider that we have a triple $(z^\star,\lambda^\star,\nu^\star)$ that satisfies the KKT conditions \eqref{KKTconditions}. We then define the set of all vectors at the point $z^\star$ that are locally consistent with the constraints by
\[
  G = \Set{ w \ | 
  \begin{array}{l}
    \frac{\partial g_{eq,i}}{\partial z}(z^\star)w = 0 \ \text{for} \ i\in\mathcal{M} \\
    \frac{\partial g_{ineq,i}}{\partial z}(z^\star)w = 0 \ \text{for} \ i\in\mathcal{P} \ \text{with} \ \nu_i^\star > 0 \\
    \frac{\partial g_{ineq,i}}{\partial z}(z^\star)w \leq 0 \ \text{for} \ i\in\mathcal{P} \ \text{with} \ \nu_i^\star = 0 \\
  \end{array}},
\]
where $w\in\mathbb{R}^n$, $\mathcal{M} = \left\lbrace1, \dots, m\right\rbrace$, and $\mathcal{P} = \left\lbrace1, \dots, p\right\rbrace$.

Using this set definition, a sufficient condition for local optimality can be stated. The condition utilises the Hessian of the Lagrangian function and checks to see if the sign of the Hessian in all directions within the set $G$, is positive. This is equivalent to saying that the Lagrangian function should be minimised in all directions that are consistent with the constraint equations.

\begin{theorem}[{\cite[Theorem 12.6]{nocedalNumericalOptimization1999}}]\label{thm:KKTsufficient}
	Suppose that the triple $(z^\star,\lambda^\star,\nu^\star)$ satisfies the KKT conditions \eqref{KKTconditions}. If
	\begin{equation*}
		w^\top\frac{\partial^2\mathcal{L}}{\partial z^2}(z^\star,\lambda^\star,\nu^\star) w > 0
	\end{equation*}
	for all $w\in G$, then $z^\star$ is a strict locally optimal solution of~\eqref{probFormulation}.
\end{theorem}

\subsection{Switched systems}\label{sec:switchedSystems}
We now review some basic properties of switched systems that will be used to define the optimisation dynamics.
Switched systems are defined by (i) a family of subsystems $h_s:\mathbb{R}^n\to\mathbb{R}^n$, where $s\in\mathcal{S}$ and $\mathcal{S}$ is an index set, and (ii) a piecewise constant switching signal $\sigma:\left[0,\infty\right)\to\mathcal{S}$, which determines the subsystem $h_s$ that is active at any given time. More formally, switched systems can be described by a differential equation of the form:
\begin{equation}\label{switchedSystem}
	\dot z = h_{\sigma(t)}(z),
\end{equation}
which has a continuous right-hand side for almost all time. We distinguish the points of discontinuity of $\sigma(t)$, also known as switching instants, by $t_1, t_2, \dots, t_k$. The number of switches that a system undergoes can be either finite or infinite~\cite{liberzonSwitchingSystemsControl2003}.

When considering solutions to switched systems of the form \eqref{switchedSystem}, it can be seen that the solution has a discontinuous derivative at the instants of switching. Due to this, we consider Carath\'eodory solutions on an interval $\left[0, t\right]$ given by
\begin{equation}
	\begin{split}
		z(t)
		&=
		z(0)
		+
		\int_0^t h_{\sigma(\tau)}(z(\tau)) \ d\tau.
	\end{split}
\end{equation}
Such solutions are absolutely continuous and satisfy \eqref{switchedSystem} for almost all $t$; see \cite{cortesDiscontinuousDynamicalSystems2008}. {A solution is called complete if it is defined for all time $t$.}

The behaviour of solutions of switched systems is complicated with the possibility of chattering and Zeno behaviours~\cite{liberzonSwitchingSystemsControl2003}. To exclude such behaviours, it is commonplace to consider an average dwell-time condition, which limits the maximum number of switches that can occur in a fixed amount of time. Letting $N_\sigma(t,\tau)$ denote the number of switches that occur on an open time interval $(\tau,t)$, where $t\geq\tau\geq 0$, we say that a system has an average dwell-time if there exist constants $N_0$ and $\tau_D$ such that
\begin{equation}\label{averageDwellTimeDefinition}
	N_\sigma(t,\tau) \leq N_0 + \frac{t-\tau}{\tau_D}.
\end{equation}
In this expression, the constant $N_0$ is the chatter-bound that describes the maximum number of switches that can occur independent of any time consideration, and the constant $\tau_D$ is the average dwell-time that determines the maximum number of switches that can occur in a single unit of time \cite{hespanhaStabilitySwitchedSystems1999}.

We now introduce the notions of a common Lyapunov function and invariance for switched systems.
\begin{definition}[Common Lyapunov function \cite{bacciottiInvariancePrincipleNonlinear2005}]
	A positive definite and differentiable function $V$ is a common weak Lyapunov function for \eqref{switchedSystem} if
	\begin{equation*}
		\nabla_z^\top V(z)h_s(z) \leq 0
	\end{equation*}
	for all $z$ in some neighbourhood of the minima of $V(z)$ and all $s\in\mathcal{S}$.
\end{definition}

\begin{definition}[Weakly invariant set \cite{mancilla-aguilarInvariancePrinciplesSwitched2011}]
	Given a family of complete trajectories $\mathcal{T}^\star$ of \eqref{switchedSystem}, a non-empty subset $M\subset\mathbb{R}^n\times\mathcal{S}$ is weakly invariant with respect to $\mathcal{T}^\star$ if for each $(\xi,\gamma)\in M$, there exists a trajectory in $\mathcal{T}^\star$ such that $z(0)=\xi,\sigma(0)=\gamma$ and $(z(t),\sigma(t))\in M$ for all time.
\end{definition}
In both definitions given above, the term \emph{weak}  reflects the fact that solutions of switched systems may not be unique. In the cases when the uniqueness of solutions is ensured, these definitions coincide with the standard definitions of Lyapunov functions and invariant sets.

\subsection{Problem statement, approach, and contributions}
The objective of this work is to propose a switched system of the form \eqref{switchedSystem} that converges to a point $z^\star$ satisfying the KKT conditions \eqref{KKTconditions} of the nonlinear optimisation problem \eqref{probFormulation}. 

The proposed method requires that the initial condition of the switched system satisfies the inequality constraints, i.e., $z(0)\in\mathcal{G}_{ineq}$. The switching law that is defined as part of the dynamics is then used to ensure that the set $\mathcal{G}_{ineq}$ is positively invariant. Note that this is similar to interior point methods that utilise barrier functions for a similar effect. Convergence of the algorithm is ensured via Lyapunov analysis by using the objective function $f$ as a common Lyapunov function between all of the possible subsystems.

Compared to existing methods, the main contributions of our work are:
{
\begin{itemize}

\item The proposed optimisation dynamics are applicable to nonconvex optimisation problems and provably converge to a solution of the optimisation problem, provided that a technical condition is satisfied. In contrast with existing primal dual dynamic solutions for nonconvex problems \cite{cherukuriSaddlePointDynamicsConditions2017}, \cite{maDistributedApproachesSolving2016} the domain of attraction is large, defined by the set $\mathcal{G}_{ineq}$.

\item We propose a novel approach to continuous-time optimisation using a switched systems framework. Compared to existing primal-dual dynamics methods, our method focuses on estimating only the primal variable without estimating Lagrange multipliers. This results in optimisation dynamics of dimension $n$, which is significantly lower than similar primal-dual methods that have a dimension of $n+m+p$.

\end{itemize}
}

\section{optimisation Dynamics}
\label{sec:optim_dyn}
This section introduces the optimisation dynamics that achieve the problem statement. We first introduce some additional notation related to the inequality constraints.

\subsection{Constraint and switching notation}
Considering the condition \eqref{KKTconditions:complimentary}, the Lagrange multipliers associated with inequality constraints satisfying $g_{ineq,i}(z^\star) < 0$ are equal to zero. When substituting these values into  \eqref{KKTconditions:gradient}, the gradients corresponding to inequality constraints that hold strictly vanish, implying that inequality constraints that hold strictly do not influence the gradient condition at a solution. When constructing a set of optimisation dynamics that seek a candidate solution $z^\star$, a complication arises from the fact that we do not know a-priori which inequality constraints will hold with equality at a solution $z^\star$.

This complication is addressed by introducing a switching law that selects an appropriate subsystem for \eqref{switchedSystem} based on the state of the inequality constraints. We now introduce some notation that will enable the definition of subsystems and a switching law based on the inequality constraints. We let $\mathcal{P} = \{1, \dots, p\}$ be the set containing all indices associated with inequality constraints. The power set $\mathbb{P}(\mathcal{P})$ contains all possible combinations of inequality constraint indices, including the empty set. The set-valued function $\mathcal{I}(z)$ is defined as the set of all inequality indices that correspond to the inequality constraints holding with equality at a given point $z$, i.e., 
\begin{equation*}
	\mathcal{I}(z)
	=
	\left\lbrace
		i\in\mathcal{P} \ | \ g_{ineq,i}(z) = 0
	\right\rbrace.
\end{equation*}
Finally, the set $\mathcal{S}$ is the set of all inequality constraint index combinations that can simultaneously and consistently hold with equality. More formally, the set is defined by
\begin{equation}\label{setSdefinition}
	\mathcal{S}
	=
	\left\lbrace
		s\in\mathbb{P}(\mathcal{P}) \ | \ \exists z\in\mathbb{R}^n \ \text{with} \ s \subseteq \mathcal{I}(z)
	\right\rbrace.
\end{equation}
In the sequel, the set $\mathcal{S}$ will form the index set for the switched systems and the switching law will be defined such that $\sigma:\left[0,\infty\right)\to\mathcal{S}$. For any given point in time, the value of $\sigma(t)\in\mathcal{S}$ will be referred to as the active set.

Now we introduce some notation related to the constraints in \eqref{probFormulation} that is used to define the right-hand vector fields of \eqref{switchedSystem}. For a given element $\mathcal{A}\in\mathcal{S}$ with cardinality $a$, we construct an indicator matrix $E_\mathcal{A}\in\mathbb{R}^{a\times p}$, which contains as rows a standard basis vector for each index in the set $\mathcal{A}$. The indicator matrix has the form
\begin{equation*}
	E_\mathcal{A}
	=
	\begin{bmatrix}
		e_i &
		\cdots &
		e_k
	\end{bmatrix}^\top,
\end{equation*}
where $i, k\in\mathcal{A}$ and each $e_j$ is the $j^{th}$ standard basis vector. For example, for $\mathcal{A} = \left\lbrace1, 3\right\rbrace$, we have $E_{\left\lbrace1, 3\right\rbrace}
	=
	\begin{bmatrix}
		e_1 &
		e_3
	\end{bmatrix}^\top$.
The indicator matrix is used to remove all inequality constraints whose indices do not appear in the set $\mathcal{A}$ via the multiplication:
\begin{equation*}
	\begin{split}
		E_\mathcal{A}
		\underbrace{
		\left[
			A_{ineq}(z)\nabla_z f + d_{ineq}
		\right]
		}_{g_{ineq}(z)} \leq 0_{a\times 1}.
	\end{split}
\end{equation*}
Moving forward, we simplify notation by collecting all equality constraints and inequality constraints with indices contained in $\mathcal{A}$ into a single vector
\begin{equation}\label{constraintsReduced}
	\begin{split}
			g_\mathcal{A}(z)
			:=
			\underbrace{
				\begin{bmatrix}
					A_{eq}(z) \\
					E_\mathcal{A}A_{ineq}(z)
				\end{bmatrix}
			}_{:= A_\mathcal{A}(z)}
			\nabla_z f
			+
			\underbrace{
				\begin{bmatrix}
					d_{eq} \\
					E_\mathcal{A}d_{ineq}
				\end{bmatrix}
			}_{:= d_\mathcal{A}}.
	\end{split}
\end{equation}
The constraint vector obtained via the switching law $g_{\sigma(t)}(z)$ at any given $t$ is called the active constraint vector.

\subsection{Technical assumption}
The optimisation dynamics requires a technical assumption to ensure that its solutions converge to a point satisfying the KKT conditions \eqref{KKTconditions}. The required technical assumption is related to the matrix
\begin{equation}\label{Bdefinition}
	\begin{split}
	 B_\mathcal{A}(z)
		:=
		\underbrace{
		\begin{bmatrix}
			\frac{\partial g_{eq}}{\partial z}(z) \\
			E_\mathcal{A}\frac{\partial g_{ineq}}{\partial z}(z)
		\end{bmatrix}
		}_{\frac{\partial g_\mathcal{A}}{\partial z}(z)}
		\underbrace{
		\begin{bmatrix}
			A_{eq}^\top(z) & A_{ineq}^\top(z)E_\mathcal{A}^\top
		\end{bmatrix}
		}_{A_\mathcal{A}^\top(z)},
	\end{split}
\end{equation}
which is obtained by multiplying the Jacobian of the active constraint vector $g_\mathcal{A}(z)$ with the matrix $A_\mathcal{A}(z)$, both defined in \eqref{constraintsReduced}. This matrix is used in the sequel to add damping-like terms to the optimisation dynamics and is required to be positive in the domain of interest.

\begin{assumption}\label{Assumption1}
	It is assumed that $ B_\mathcal{A}(z)$ is invertible and the symmetric component of $ B_\mathcal{A}(z)$ is uniformly positive definite for all $z\in\mathcal{G}_{ineq}$ and $\mathcal{A}\in\mathcal{S}$. Moreover, there exist constants $\beta_{\mathcal{A}}^1$ and $\beta_{\mathcal{A}}^2 > 0$ such that
\begin{equation}\label{Bpositive}
		\begin{split}
			\frac12\left[ B_\mathcal{A}(z) +  B_\mathcal{A}^\top(z)\right]  
			\geq
			\frac12\beta_{\mathcal{A}}^1 &B_\mathcal{A}(z) B_\mathcal{A}^\top(z) \geq
			\beta_{\mathcal{A}}^2 I_{m+a} > 0.
		\end{split}
	\end{equation}
\end{assumption}

{
\begin{remark}
	Assumption \ref{Assumption1} is made primarily for technical reasons and is the primary limitation of the proposed approach. In the case of quadratic problems subject to linear constraints, the condition is equivalent to linear independence of the constraints. However, in the nonconvex case the interpretation of this assumption requires further investigation. Nevertheless, in Section~\ref{sec:examples}, we provide two nonconvex examples, motivated from practice, that satisfy this assumption.
\end{remark}
}

\begin{remark}
	The domain of Assumption \ref{Assumption1} is limited to $z\in\mathcal{G}_{ineq}$. In the sequel, it will be shown that this domain is positively invariant along the solutions to the proposed optimisation dynamics, which implies that the technical assumption must only hold on this region. 
\end{remark}

\begin{remark}\label{rem:fullRankJacobian}
	As the matrix $B_\mathcal{A}(z)$ is invertible, it is necessary that the Jacobian $\frac{\partial g_\mathcal{A}}{\partial z}(z)$ has full row rank for all $z\in\mathcal{G}_{ineq}$. This implies that the constraints must be linearly independent on the same set, {which is a necessary condition for the first-order optimality conditions, revised in Section~\ref{sec:probStatement:KKT}}.
\end{remark}

\begin{lemma}\label{lem:BAinequality}
	From Assumption \ref{Assumption1}, it follows that
	\begin{equation*}
		\begin{aligned}
			\frac12\left[ B_\mathcal{A}^{-\top}(z) +  B_\mathcal{A}^{-1}(z)\right]
			&\geq
			\frac12\beta_{\mathcal{A}}^1 I_{m+a} \\
			&\geq
			\beta_{\mathcal{A}}^2  B_\mathcal{A}^{-1}(z) B_\mathcal{A}^{-\top}(z)
			>
			0.
		\end{aligned}
	\end{equation*}
\end{lemma}

\begin{proof}
	Left and right multiplying \eqref{Bpositive} by $ B_\mathcal{A}^{-1}(z)$ and $ B_\mathcal{A}^{-\top}(z)$, respectively, yields the desired result.
\end{proof}

The Jacobian of the active constraint vector $g_\mathcal{A}(z)$, i.e., $\frac{\partial^\top g_\mathcal{A}}{\partial z}(z)$, has dimension $n\times (m+a)$. As discussed in Remark~\ref{rem:fullRankJacobian}, Assumption~\ref{Assumption1} implies that $\frac{\partial^\top g_\mathcal{A}}{\partial z}(z)$ has full row rank for all $z\in\mathcal{G}_{ineq}$ and $\mathcal{A}\in\mathcal{S}$. For each $\mathcal{A}\in\mathcal{S}$, we define a full rank left annihilator for $\frac{\partial^\top g_\mathcal{A}}{\partial z}(z)$ denoted by $G_\mathcal{A}^\perp(z)\in\mathbb{R}^{(n-m-a)\times n}$, which satisfies 
\begin{equation}\label{gPerpDef}
	\begin{split}
		G_\mathcal{A}^\perp(z)
		\underbrace{
	\begin{bmatrix}
		\frac{\partial^\top g_{eq}}{\partial z}(z) & \frac{\partial^\top g_{ineq}}{\partial z}(z)E_\mathcal{A}^\top(z)
	\end{bmatrix}
		}_{\frac{\partial^\top g_\mathcal{A}}{\partial z}(z)}
	&=
	0_{(n-m-a)\times (m+a)}
	\end{split}
\end{equation}
for all $z\in\mathcal{G}_{ineq}$. Note that the dimension of $G_\mathcal{A}^\perp(z)$ will vary depending on the cardinality of any given $\mathcal{A}\in\mathcal{S}$.

\subsection{Subsystem dynamics}
Using the notation introduced throughout this section, the dynamics of each subsystem that forms the right-hand side vector field of \eqref{switchedSystem} is given by
\begin{equation}\label{optimisationDynamicsOriginal}
	\begin{aligned}
		&\dot z
		=
		h_\mathcal{A}(z) \\
		&=
		-\left[\kappa_1A_\mathcal{A}^{\top}(z) B_\mathcal{A}^{-1}(z)A_\mathcal{A}(z) + \kappa_2G_\mathcal{A}^{\perp\top}(z)G_\mathcal{A}^\perp(z)\right]  \nabla_z f(z) 
  \\
  &\hspace{4mm}- \kappa_1A_\mathcal{A}^{\top}(z) B_\mathcal{A}^{-1}(z)d_\mathcal{A}
	\end{aligned}
\end{equation}
all $\mathcal{A}\in\mathcal{S}$, where $A_\mathcal{A}(z)$ and $d_\mathcal{A}$ are defined in \eqref{constraintsReduced}, $ B_\mathcal{A}(z)$ is defined in \eqref{Bdefinition}, $G_\mathcal{A}^\perp(z)$ satisfies \eqref{gPerpDef}, and $\kappa_1, \kappa_2\in\mathbb{R}$ are positive  tuning parameters. By considering the definition \eqref{constraintsReduced}, the dynamics \eqref{optimisationDynamicsOriginal} can be equivalently written  as 
\begin{equation}\label{optimisationDynamicsFactored}
	\begin{split}
		\dot z
		=
		-\kappa_1 A_\mathcal{A}^{\top}(z) B_\mathcal{A}^{-1}(z)g_\mathcal{A}(z) 
		-\kappa_2 G_\mathcal{A}^{\perp\top}(z)G_\mathcal{A}^\perp(z)\nabla_z f(z).
	\end{split}
\end{equation}

The intuition behind the proposed dynamics is evident from~\eqref{optimisationDynamicsFactored}. The second term on the right side of \eqref{optimisationDynamicsFactored} is a gradient descent term on the objective function $f(z)$ in directions perpendicular to the constraints, whereas the first term on the right side of \eqref{optimisationDynamicsFactored} is used to drive the active constraint vector $g_\mathcal{A}(z)$ towards the origin using the positive definiteness of the matrix $B_\mathcal{A}(z)$ (see Assumption \ref{Assumption1}). While it is clear that this is desirable for equality constraints, it is not immediately clear why this should be required for inequality constraints. In the sequel, the switching law $\sigma(t)\in\mathcal{S}$ will be introduced to ensure that the inequality constraints are only included in the active constraint vector when they hold with equality. In that case, the role of the first term is to ensure that the inequality constraints are not violated along the forward solution of the switched system.

\subsection{Switching law}
The dynamics of each subsystem of the optimisation dynamics is given by \eqref{optimisationDynamicsOriginal}. Next, to complete the definition of the switched system, we introduce a switching law $\sigma:\left[0,\infty\right)\to\mathcal{S}$ that determines which subsystem should be integrated at any given point in time.

Before formally introducing the switching law, we provide a qualitative example that demonstrates the intuition behind the design of the switching law. Figure~\ref{switchingExample} shows an example optimisation problem, where the objective function is given by Rosenbrock's function and a single linear inequality constraint constrains the solution space. The index set is given by $\mathcal{S} = \left\lbrace\{\emptyset\}, \{1\}\right\rbrace$ and the switching law $\sigma(t)$ can take the possible values $\{1\}$ or $\{\emptyset\}$. The inequality constraint holds with equality at both points $z_a$ and $z_b$, and two possible flow vectors are available that correspond to the two subsystems. At point $z_a$, integration of the subsystem $h_{\{\emptyset\}}(z_a)$ would violate the constraint, whereas $h_{\{1\}}(z_a)$ ensures that the constraint is held with equality, so the switching law should be such that $\sigma(t) = \{1\}$. Conversely, at point $z_b$ integrating subsystem $h_{\{\emptyset\}}(z_b)$ would move away from the constraint boundary and decrease the objective function, whereas $h_{\{1\}}(z_b)$ would ensure the constraint holds with equality. In this case, the switching law should result in $\sigma(t) = \{\emptyset\}$.
\begin{figure}[ht!]
	\centering{}
	\includegraphics[width=0.90\columnwidth]{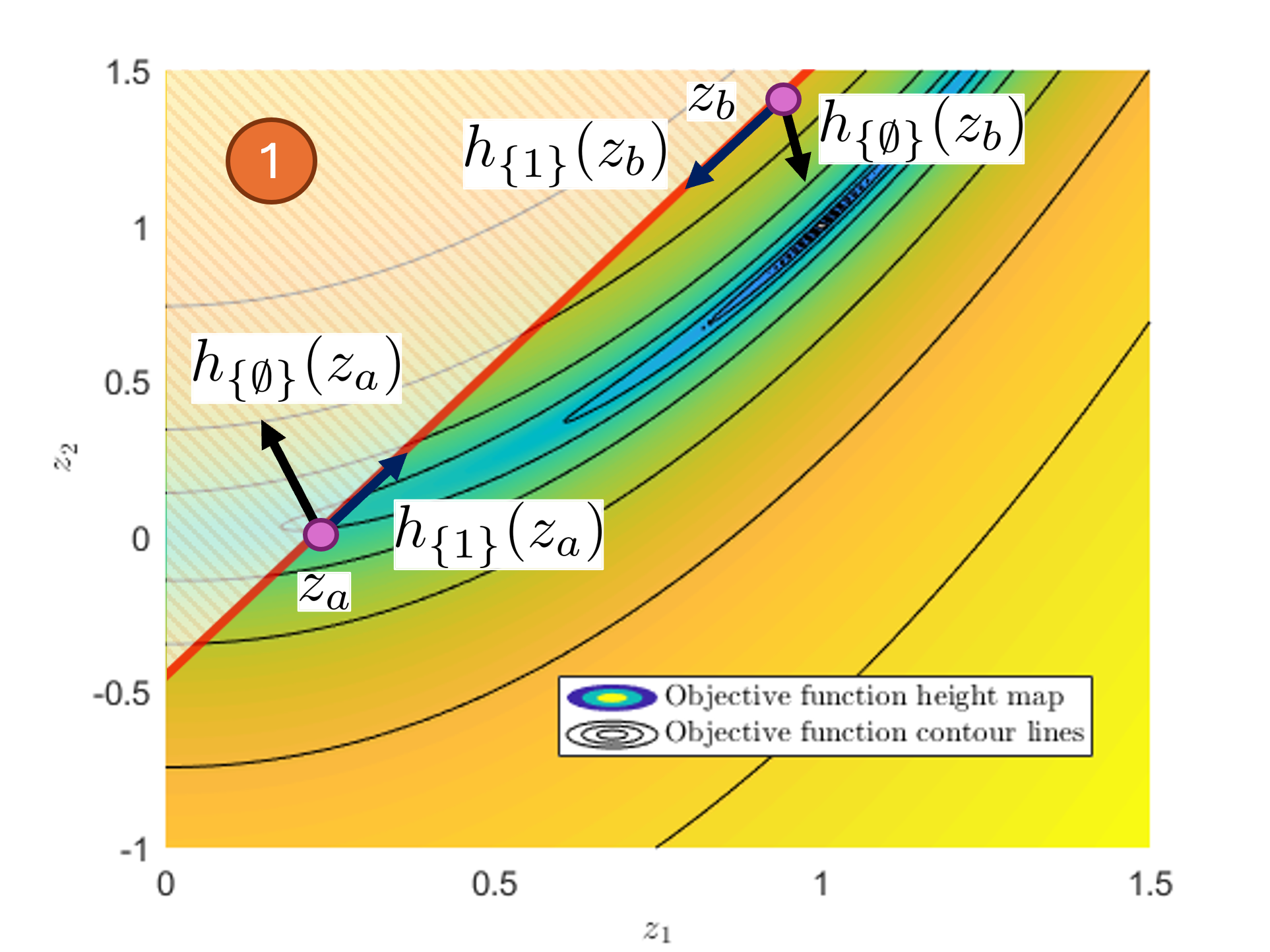}
	\caption{Height map of Rosenbrock's function with an inequality constraint. At the point $z_a$, the steepest descent direction intersects the constrained space, whereas at the point $z_b$, the steepest descent gradient does not intersect the constrained space.}
	\label{switchingExample}
\end{figure}

The switching law for the optimisation dynamics is given in Algorithm \ref{alg:switching}\footnote{{As the optimisation dynamics operate in continuous time and the solution converges asymptotically, no terminal condition of the switching law is defined.}}. For each constraint satisfying $g_{ineq,i}(z) = 0$ that is not in the active set, $i\notin\mathcal{I}(z(t))\backslash\sigma(t)$, the gradient of the constraint along the current subsystem flow is checked to determine if integration of $\dot z = h_{\sigma(t)}(z)$ would violate the constraint. If an increase is detected via the evaluation of $\dot g_{ineq,i}(t)$, the dynamics are switched such that the constraint is added to the updated set $\sigma(t^+)$, which causes the constraint to be held with equality along the forward solution. A similar algorithm is applied to remove constraints from the set. For each constraint $i$ in the set $\sigma(t)$, the gradient of the constraint along the flow that would result from switching to $\sigma(t)\backslash\left\lbrace i\right\rbrace$, is considered. If integration of $\dot z = h_{\sigma(t)\backslash\left\lbrace i\right\rbrace}(z)$ would result in a decrease in the constraint value, the switching $\sigma(t^+) = \sigma(t)\backslash\left\lbrace i\right\rbrace$ occurs. When removing a constraint, however, a minimum time $\delta T > 0$ must have elapsed since a constraint was last removed. This restriction ensures an average dwell-time for the dynamics to avoid any potential Zeno or chattering behaviours in the solution.
\begin{algorithm}
	\caption{Switching law}\label{alg:switching}
	\begin{algorithmic}
		\State \textbf{Inputs:} $t, z(t), \sigma(t)$
		\State \textbf{Outputs:} $\sigma(t^+)$
		\State \textbf{Initialisation:} $\sigma(0) = \mathcal{I}(z(0))$
		\State Check if the set $\sigma(t)$ should be updated to include an additional inequality constraint index.
		\For{$i\in\mathcal{I}(z(t))\backslash\sigma(t)$}
			\If{$\frac{\partial g_{ineq,i}}{\partial z}(z(t))h_{\sigma(t)}(z(t)) \geq 0$}
				\State $\sigma(t^+) = \sigma(t) \cup \left\lbrace i\right\rbrace$
			\EndIf
		\EndFor
		\State Check if the set $\sigma(t)$ should be updated to remove an  inequality constraint index.
		\For{$i\in\sigma(t)$}
			\If{$\frac{\partial g_{ineq,i}}{\partial z}(z(t))h_{\sigma(t)\backslash \left\lbrace i\right\rbrace}(z(t)) < 0$ \textbf{and} $t > T_s + \delta T$}
				\State $\sigma(t^+) = \sigma(t) \backslash \left\lbrace i\right\rbrace$
				\State $T_s = t$
			\EndIf
		\EndFor
	\end{algorithmic}
\end{algorithm}

\begin{remark}
It will be shown in subsequent analysis that the switching law has the effect of rendering the set $\mathcal{G}_{ineq}$ positively invariant along solutions of the switched system. With this in mind, the proposed approach shares similarities with both interior point methods and active set methods for optimisation. Interior point methods utilise barrier functions to ensure positive invariance of $\mathcal{G}_{ineq}$, whereas active set methods switch between different combinations of inequality constraints being enforced with equality to reach a solution \cite{nocedalNumericalOptimization1999}. Our proposed optimisation dynamics includes aspects from both of the methods mentioned above.  
\end{remark}


\section{Stability analysis}\label{sec:stab_anal}
In this section, we study the stability properties of the switched system, formed by the subsystem dynamics \eqref{optimisationDynamicsOriginal} with switching law given in Algorithm~\ref{alg:switching}. The underlying principle for the stability analysis is to use the objective function $f$ as a common ISS-Lyapunov function for the switched system. Combining this property with the behaviour of the switching law results in convergence of the solution of the switched system to an equilibrium point $z^\star$ satisfying the KKT conditions~\eqref{KKTconditions}.

Before proceeding to the technical details of the proof, we first state the main result of the paper. The justification of this claim is developed throughout the remainder of this section.
\begin{theorem}\label{thm:convergence}
	Consider the switched system formed by combining the subsystem dynamics~\eqref{optimisationDynamicsOriginal} with the switching law given in Algorithm~\ref{alg:switching}. If Assumption \ref{Assumption1} is satisfied and $z(0) \in \mathcal{G}_{ineq}$, then the solution $z(t)$ of the switched system converges to a point $z^\star\in\mathbb{R}^n$ for which there exist corresponding Lagrange multipliers $\lambda^\star$ and $\nu^\star$, such that $(z^\star, \lambda^\star, \nu^\star)$ satisfies the KKT conditions \eqref{KKTconditions}.
\end{theorem}

The proof of this claim follows in Section \ref{sec:mainResultProof}.

\subsection{Positive invariance of feasibility set}
Verification of the claimed convergence properties starts by establishing positive invariance of the constraint sets $\mathcal{G}_{eq}$, $\mathcal{G}_{ineq}$, and $\Gamma$ along forward solutions of the switched system. By construction, the subsystem dynamics \eqref{optimisationDynamicsFactored} are such that the active constraint vector \eqref{constraintsReduced} converges uniformly to the origin. Inspection of the switching law reveals that inequality constraint indices can only be added or removed from the set $\sigma(t)$ when they are equal to zero, which results in the desired invariance property. As the switching law operates by adding or removing indices from the set $\sigma(t)$, it must be additionally shown that the set is always contained within the index set $\mathcal{S}$, so that there always exists a corresponding subsystem. With these properties at hand, the switching law ensures that a minimum time $\delta T$ must elapse before inequality constraint indices are removed from the set $\sigma(t)$. This restriction enforces an average dwell-time that excludes undesirable behaviours from the solutions of the switched system.
\begin{proposition}\label{prop:feasibleSetPosInvar}
Consider the switched system formed by combining the subsystem dynamics \eqref{optimisationDynamicsOriginal} with the switching law given in Algorithm~\ref{alg:switching}. The following properties hold for any solution of the switched system:
\begin{enumerate}
	\item\label{prop:Existence:averageDwell} The solution has an average dwell-time of the form \eqref{averageDwellTimeDefinition} with $N_0 = p$ and $\tau_D = \frac{\delta T}{2}$.

	\item\label{prop:subsystemBehaviour:gA}  Considering the time interval between two consecutive switches $t\in\left[t_k, t_{k+1}\right)$, the active constraint vector converges to the origin at the rate
	\begin{equation}\label{gaSolution}
		g_{\sigma(t_k)}(t) = g_{\sigma(t_k)}(t_k)e^{-\kappa_1 t},
	\end{equation}
	which can be separated into equality and inequality constraints by
\begin{subequations}\label{gaSolution_separate}
\begin{align}
			g_{eq}(t) &= g_{eq}(t_k)e^{-\kappa_1 t} \label{gaSolution_separate:equality}\\
			E_{\sigma(t_k)}g_{ineq}(t) &= E_{\sigma(t_k)}g_{ineq}(t_k)e^{-\kappa_1 t}.\label{gaSolution_separate:inequality}
\end{align}
\end{subequations}

	\item\label{prop:restrictedDynamics:switching} Let $\left[0, T_L\right)$, where $T_L\leq\infty$, be a time interval on which a solution of the switched system exists. The set $\sigma(t)$ is positively invariant on $\mathcal{S}$ and a subset of the index set $\mathcal{I}(z(t))$ for all $t\in\left[0, T_L\right)$. That is,  $\sigma(t)\subseteq\mathcal{I}(z(t))\subseteq\mathcal{S} \ \forall \ t\in\left[0, T_L\right)$.
	
	\item\label{prop:restrictedDynamics:positiveInvariance}  Solutions to the switched system are positively invariant on the sets $\mathcal{G}_{eq}$, $\mathcal{G}_{ineq}$, and $\Gamma$.  
\end{enumerate}
\end{proposition}

\begin{proof}
Consider Claim \ref{prop:Existence:averageDwell} and note that the switching law in Algorithm \ref{alg:switching} has no restriction on how quickly inequality constraint indices can be added to the active set. As there are $p$ inequality constraints, the switching law has a chatter bound of $N_0 = p$. By definition, inequality indices can only be removed from the active set after a time period of $\delta T$ has elapsed since the last index was removed. There is no time-based restriction on how quickly a constraint index can be re-added to the active set, so in a time period of $\delta T$, there can be at most two switches. This maximal switching rate is described by the average dwell-time bound $\tau_D = \frac{\delta T}{2}$.

To verify claim \ref{prop:subsystemBehaviour:gA}, we take the time derivative of the vector $g_{\sigma(t_k)}(z)$ along the dynamics \eqref{optimisationDynamicsFactored}, replacing $\mathcal{A}$ with ${\sigma(t_k)}$. The resulting constraint dynamics are given by
		\begin{equation}
  \label{constrait_dynamics}
			\begin{split}
				\dot g_{\sigma(t_k)}(z)
				=&
				\frac{\partial g_{\sigma(t_k)}}{\partial z}(z)\dot z_{\sigma(t_k)} \\
				=&
				-\kappa_1\frac{\partial g_{\sigma(t_k)}}{\partial z}(z) A_{\sigma(t_k)}^{\top}(z) B_{\sigma(t_k)}^{-1}(z)g_{\sigma(t_k)}(z) \\
				&-\kappa_2\frac{\partial g_{\sigma(t_k)}}{\partial z}(z) G_{\sigma(t_k)}^{\perp\top}(z)G_{\sigma(t_k)}^\perp(z)\nabla_z f(z)
				 \\
				=&
				-\kappa_1 g_{\sigma(t_k)}(z),
			\end{split}
		\end{equation}
		where we have used the definitions of $G_\mathcal{A}^\perp(z)$ from \eqref{gPerpDef} and $ B_\mathcal{A}(z)$ from \eqref{Bdefinition}, replacing $\mathcal{A}$ with ${\sigma(t_k)}$. Then, \eqref{constrait_dynamics} can be decomposed as
		\begin{equation*}
			\begin{split}
				\dot g_{eq}(z)
				&=
				-\kappa_1 g_{eq}(z) \\
				E_{\sigma(t_k)}\dot g_{ineq}(z)
				&=
				-\kappa_1 E_{\sigma(t_k)}g_{ineq}(z),
			\end{split}
		\end{equation*}
		whose solutions yield \eqref{gaSolution_separate}.
		
Now we consider Claim \ref{prop:restrictedDynamics:switching}: Assume that at some time $t_k$, we have $\sigma(t_k)\subseteq\mathcal{I}(z(t_k))$ and $E_{\sigma(t_k)}g_{ineq}(t_k) = 0_{a\times 1}$. Then, from Claim~\ref{prop:subsystemBehaviour:gA}, we have that $E_{\sigma(t_k)}g_{ineq}(t) = E_{\sigma(t_k)}g_{ineq}(t_k)e^{-\kappa_1 t}$, which implies that  $E_{\sigma(t_k)}g_{ineq}(t) = 0_{a\times 1}$ for all $t\in\left[t_{k},t_{k+1}\right)$. We now consider the behaviour of constraints under the addition and removal of constraint indices from the active set at time $t_{k+1}$. 
	\begin{itemize}
		\item \textbf{Case 1 - An index is added to the active set:} In the case that an index $i$ is added to the active set, it should satisfy $i\in\mathcal{I}(z(t_{k+1}))\backslash\sigma(t_{k+1})$, which implies that $\sigma(t_{k+1}^+) = \sigma(t_{k+1}) \cup \left\lbrace i\right\rbrace\subseteq\mathcal{I}(z(t_{k+1}))$. \\
		\item \textbf{Case 2 - An index is removed from the active set:} As $\sigma(t_{k+1})\subseteq\mathcal{I}(z(t_{k+1}))$, we have that $\sigma(t_{k+1}^+) = \sigma(t_{k+1})\backslash\left\lbrace i\right\rbrace\subset\mathcal{I}(z(t_{k+1}))$.
	\end{itemize}
	As the switched signal is initialised with $\sigma(0)=\mathcal{I}(z(0))$ and $\mathcal{I}(z)\in\mathcal{S}$ for all $z$, the claim follows by induction.

	Now we consider Claim~\ref{prop:restrictedDynamics:positiveInvariance}: Considering the set $\mathcal{G}_{eq}$, the equality constraints evolve according to $g_{eq}(t) = g_{eq}(t_k)e^{-\kappa_1 t}$ for all subsystems by Claim~\ref{prop:subsystemBehaviour:gA}. Therefore, if $z(0)\in\mathcal{G}_{eq}$ and $g_{eq}(t) = 0_{m\times 1}$ for all $t\in\left[0, T_L\right)$, then $\mathcal{G}_{eq}$ is positively invariant.	Consideration of the inequality constraints is more complicated due to the switching law. We need to show that the set $\mathcal{G}_{ineq}$ is preserved under the addition and removal of indices from the active set and that the set cannot be violated by any subsystem dynamics. To this end, consider that at some $t_k$, we have $z(t_k)\in\mathcal{G}_{ineq}$.
	\begin{itemize}
		\item \textbf{Case 1 - An index is added to the active set:} In the case that an index $i$ is added to the active set, it should satisfy $i\in\mathcal{I}(z(t_k))\backslash\sigma(t_k)$, which implies that $g_{ineq,i}(z(t_k)) = 0$. It follows from Claim~\ref{prop:subsystemBehaviour:gA} that $E_{\sigma(t_k)\cap\left\lbrace i\right\rbrace}g_{ineq}(z(t)) = 0_{a\times 1}$ for all $t\in\left[t_k, t_{k+1}\right)$ by \eqref{gaSolution_separate:inequality}. 
		
		\item \textbf{Case 2 - An index is removed from the active set:} For an index to be removed from the active set, it must satisfy $\frac{\partial g_{ineq,i}}{\partial z}(z)h_{\sigma(t_k)\backslash \left\lbrace i\right\rbrace}(z) < 0$. Therefore, after a switch event, the constraint $g_{ineq,i}(t)$ strictly decreases, which is consistent with the constraint inequality. From \eqref{gaSolution_separate:inequality}, the remaining inequality constraints satisfy $E_{\sigma(t)\backslash\left\lbrace i\right\rbrace}g_{ineq}(z(t)) = 0_{a\times 1}$ for all $t\in\left[t_{k}, t_{k+1}\right)$. 
		
		\item \textbf{Case 3 - A constraint is violated along the subsystem dynamics:} Now consider the possibility that an inequality constraint $g_{ineq,i}$ with $i\notin\sigma(t_k)$ is violated along the dynamics of subsystem $\sigma(t_k)$ with consecutive switching time $t_k, t_{k+1}$. For the sake of contradiction, assume that $g_{ineq,i}(t_{k+1}) > 0$. By continuity of $z(t)$ and the fact that $z(t_k)\in\mathcal{G}_{ineq}$, there must exist a time $T\in\left[t_k, t_{k+1}\right]$ such that $g_{ineq,i}(z(T)) = 0$ and $\frac{\partial g_{ineq,i}}{\partial z}(z(T))h_{\sigma(T)}z(T) \geq 0$. Such a point, however, would trigger a switching event implying that there is a switching event between $t_k$ and $t_{k+1}$.
	\end{itemize}
	It follows by induction that $\mathcal{G}_{ineq}$ is positively invariant. Finally, $\Gamma$ is positively invariant due to the positive invariance of $\mathcal{G}_{eq}$ and $\mathcal{G}_{ineq}$.
\end{proof}

\subsection{Existence, uniqueness, and boundedness of solutions}
In Proposition~\ref{prop:feasibleSetPosInvar}, it was verified that the sets $\mathcal{G}_{eq}$, $\mathcal{G}_{ineq}$, and $\Gamma$ are positively invariant along any forward solution to the switched system. So far, however, we have not yet verified that the solution of the system will exist for all time or it will be unique. In this section, we establish that the solutions of the switched system are bounded, which implies that they exist and are unique. We first show that the solutions of each subsystem between switching instants are bounded. As the Carath\'eodory solution to the switched system is the composition of the solutions of the individual subsystems, it follows that the full solution is similarly bounded.

Before verifying the solutions of each subsystem between switching instants are bounded, we require a technical lemma. The following lemma establishes that the matrix formed by concatenating $A_\mathcal{A}(z)$ and $G_{\mathcal{A}}^\perp(z)$ is full rank. This property is used in subsequent computations to ensure that each subsystem has some damping.
\begin{lemma}\label{lem:AGmatrix}
	Under Assumption \ref{Assumption1}, the matrix
	\begin{equation}\label{AGmatrix}
		\begin{bmatrix}
			A_\mathcal{A}(z) \\ G_\mathcal{A}^\perp(z)
		\end{bmatrix}
		\in\mathbb{R}^{n\times n}
	\end{equation}
	is full rank for all $z\in\mathcal{G}_{ineq}$ and $\mathcal{A}\in\mathcal{S}$. Consequently, there exists a $\phi_\mathcal{A} > 0$ such that
	\begin{equation}\label{AG_matrixSquared}
		\begin{bmatrix}
			A_\mathcal{A}(z) \\ G_\mathcal{A}^\perp(z)
		\end{bmatrix}^\top
		\begin{bmatrix}
			A_\mathcal{A}(z) \\ G_\mathcal{A}^\perp(z)
		\end{bmatrix}
		>
		\phi_\mathcal{A} I
	\end{equation}
	on the same domain.
\end{lemma}

\begin{proof}
	First, we verify that \eqref{AGmatrix} is full rank. For the sake of contradiction, assume that \eqref{AGmatrix} is not invertible. This implies that there the exist some non-zero vector $\left[v_1^\top \ v_2^\top\right]\in\mathbb{R}^n$ with $v_1\in\mathbb{R}^{m+a}, v_2\in\mathbb{R}^{n-m-a}$ such that
	\begin{equation}\label{rankProof1}
		\begin{bmatrix}
			v_1^\top & v_2^\top
		\end{bmatrix}
		\begin{bmatrix}
			A_\mathcal{A}(z) \\ G_\mathcal{A}^\perp(z)
		\end{bmatrix}
		=
		v_1^\top A_\mathcal{A}(z) + v_2^\top G_\mathcal{A}^\perp(z) = 0_{1\times n}.
	\end{equation}
	Right multiplying the above equation by $\frac{\partial^\top g_\mathcal{A}}{\partial z}(z)$ removes the second term, resulting in
	\begin{equation*}
		v_1^\top A_\mathcal{A}(z)\frac{\partial^\top g_\mathcal{A}}{\partial z}(z) = 0_{1\times (m+a)}.
	\end{equation*}
	By Assumption \ref{Assumption1}, this can only be true if $v_1 = 0_{(m+a)\times 1}$. Considering \eqref{rankProof1}, $v_1 = 0_{(m+a)\times 1}$ implies that we must have $v_2$ zero as $G_\mathcal{A}^\perp(z)$ has full row rank. This leads to a contradiction. Therefore, we conclude that \eqref{AGmatrix} is invertible. The identity \eqref{AG_matrixSquared} follows from \eqref{AGmatrix} being full rank on the closed domain $\mathcal{G}_{ineq}$.
\end{proof}

Boundedness of the solution to each subsystem can now be established. To provide some intuition of the proof, consider the subsystem dynamics~\eqref{optimisationDynamicsOriginal}. Notice that the first term is a gradient descent with respect to the objective function, and note that the second term appears as a disturbance related to the constant vector $d_\mathcal{A}$. Our approach to verify that the solution of each subsystem is bounded is to use the objective function as a candidate ISS-Lyapunov function. This results in a positively invariant sub-level set of $f$ in which the solution must be contained.

\begin{proposition}\label{prop:subsystemBehaviour}
	Consider the subsystem dynamics \eqref{optimisationDynamicsOriginal} on the interval $\left[t_k, t_{k+1}\right)$, where $t_k$ and $t_{k+1}$ are any consecutive switching instants such that $0 \leq t_k \leq t_{k+1} \leq \infty$. If $z(t_k)\in\mathcal{G}_{ineq}$, the solutions of the subsystem \eqref{optimisationDynamicsOriginal} are positively invariant on any sub-level set of $f$ containing
	\begin{equation}\label{FsetDefinitions}
		\begin{split}
			\mathcal{F}_\mathcal{A}
			&=
			\left\lbrace
				z\in\mathbb{R}^n \ \bigg| \right. \\
				&
				\left. 
				||\nabla_z f(z)||^2 \leq \frac{\kappa_1}{\beta_{\mathcal{A}}^2\phi_\mathcal{A}\min\left\lbrace\frac12\kappa_1\beta_{\mathcal{A}}^1, \kappa_2\right\rbrace}
				||d_\mathcal{A}||^2
			\right\rbrace,
		\end{split}
	\end{equation}
	where $d_\mathcal{A}$ is defined in \eqref{constraintsReduced}, $\beta_{\mathcal{A}}^1$ and $\beta_{\mathcal{A}}^2$ in \eqref{Bpositive},  $\phi_\mathcal{A}$ in \eqref{AG_matrixSquared}, and $t\in\left[t_k, t_{k+1}\right)$.
\end{proposition}

\begin{proof}
	Consider the objective function $f$ as a candidate ISS Lyapunov function. Taking the time derivative of $f$ along the trajectories of \eqref{optimisationDynamicsOriginal} results in
	\begin{equation}
 \label{eq:time_derv_f}
		\begin{split}
			\dot f
			\leq&
			-\frac12\kappa_1\nabla_z^\top fA_\mathcal{A}^{\top}\left\lbrace B_\mathcal{A}^{-1} +  B_\mathcal{A}^{-\top}\right\rbrace A_\mathcal{A}\nabla_z f \\
			&-\kappa_2\nabla_z^\top f G_\mathcal{A}^{\perp\top}G_\mathcal{A}^\perp\nabla_z f \\
			&+\frac{1}{2\epsilon}\kappa_1\nabla_z^\top f A_\mathcal{A}^{\top} B_\mathcal{A}^{-1} B_\mathcal{A}^{-\top}A_\mathcal{A}\nabla_z f
			+ \frac{\epsilon}{2}\kappa_1d_\mathcal{A}^\top d_\mathcal{A},
		\end{split}
	\end{equation}
	where $\epsilon$ is a constant from the application of Young's inequality, and the arguments of the functions have been dropped for brevity. From Proposition~\ref{prop:feasibleSetPosInvar}.\ref{prop:restrictedDynamics:positiveInvariance}, the set $\mathcal{G}_{ineq}$ is positively invariant, which implies that Lemma~\ref{lem:BAinequality} holds on the considered time interval. Then, applying Lemma~\ref{lem:BAinequality} to the first-term of \eqref{eq:time_derv_f} yields
	\begin{equation*}
		\begin{split}
			\dot f
			\leq&
			-\frac14\kappa_1\nabla_z^\top fA_\mathcal{A}^{\top}\left\lbrace B_\mathcal{A}^{-1} +  B_\mathcal{A}^{-\top}\right\rbrace A_\mathcal{A}\nabla_z f \\
			&-\frac12\kappa_1\left[\beta_{\mathcal{A}}^2 - \frac{1}{\epsilon}\right]\nabla_z^\top fA_\mathcal{A}^{\top} B_\mathcal{A}^{-1} B_\mathcal{A}^{-\top}(z)A_\mathcal{A}\nabla_z f \\
			&-\kappa_2\nabla_z^\top f G_\mathcal{A}^{\perp\top}G_\mathcal{A}^\perp\nabla_z f 
			+ \frac{\epsilon}{2}\kappa_1d_\mathcal{A}^\top d_\mathcal{A}.
		\end{split}
	\end{equation*}
	Setting $\epsilon = (\beta_{\mathcal{A}}^2)^{-1}$ and again applying Lemma \ref{lem:BAinequality} results in
	\begin{equation*}
		\begin{split}
			\dot f
			\leq&
			-\frac14\kappa_1\nabla_z^\top fA_\mathcal{A}^{\top}\left\lbrace B_\mathcal{A}^{-1} +  B_\mathcal{A}^{-\top}\right\rbrace A_\mathcal{A}\nabla_z f \\
			&-\kappa_2\nabla_z^\top f G_\mathcal{A}^{\perp\top}G_\mathcal{A}^\perp\nabla_z f
			+ \frac12(\beta_{\mathcal{A}}^2)^{-1}\kappa_1d_\mathcal{A}^\top d_\mathcal{A} \\
			\leq&
			-\min\left\lbrace\frac14\kappa_1\beta_{\mathcal{A}}^1, \kappa_2\right\rbrace
			\nabla_z^\top f
			\begin{bmatrix}
				A_\mathcal{A}^{\top} & G_\mathcal{A}^{\perp\top}
			\end{bmatrix}
			\begin{bmatrix}
				A_\mathcal{A} \\ G_\mathcal{A}^{\perp}
			\end{bmatrix} \\
			&\times\nabla_z f(z)
			+ \frac12(\beta_{\mathcal{A}}^2)^{-1}\kappa_1d_\mathcal{A}^\top d_\mathcal{A} \\
			\leq&
			-\min\left\lbrace\frac14\kappa_1\beta_{\mathcal{A}}^1, \kappa_2\right\rbrace
			\phi_\mathcal{A}
			||\nabla_z f||^2
			+ \frac12(\beta_{\mathcal{A}}^2)^{-1}\kappa_1||d_\mathcal{A}||^2,
		\end{split}
	\end{equation*}
	where Lemma \ref{lem:AGmatrix} has been used in the last line.
	It follows that $\dot f(z)$ is strictly decreasing for any $z$ satisfying
	\begin{equation*}
		\begin{split}
			||\nabla_z f(z)||^2
			>
			\frac{\kappa_1}{2\beta_{\mathcal{A}}^2\phi_\mathcal{A}\min\left\lbrace\frac14\kappa_1\beta_{\mathcal{A}}^1, \kappa_2\right\rbrace}||d_\mathcal{A}||^2.
		\end{split}
	\end{equation*}
	As $f(z)$ and $||\nabla_z f(z)||$ are radially unbounded, any sub-level set of $f(z)$ that contains $\mathcal{F}_\mathcal{A}$ is positively invariant along the solutions of \eqref{optimisationDynamicsOriginal} on the time interval $\left[t_k, t_{k+1}\right)$.
\end{proof}

It has now been established that the solution of any given subsystem between switching instants must be contained within a sub-level set of $f$ by Proposition~\ref{prop:subsystemBehaviour}. Our attention now turns to verifying that the solution to the switched system is similarly bounded. This property follows from the fact that $f$ is used as a common ISS-Lyapunov function for all subsystems. Therefore, there exists a sub-level set of $f$ that contains $\mathcal{F}_\mathcal{A}$, defined in \eqref{FsetDefinitions}, for all $\mathcal{A}\in\mathcal{S}$. The existence of such a set implies that any forward solution to the switched system is bounded. A consequence of boundedness of the solution is that the solution is both unique and exists for all time. 

\begin{proposition}\label{prop:existence_uniqueness}
	Consider the switched system formed by combining the subsystem dynamics~\eqref{optimisationDynamicsOriginal} with the switching law given in Algorithm \ref{alg:switching}. If $z(0)\in\mathcal{G}_{ineq}$, then the solution of the switched system satisfies the following properties:
	\begin{enumerate}
	\item\label{prop:Existence:invariantSet} The forward solution is positively invariant on any sub-level set of $f$ that contains
\begin{equation}\label{FsetDefinitions2}
		\begin{split}
			&\mathcal{F}
			=
			\left\lbrace
				z\in\mathbb{R}^n \ \bigg| \right. \\
				&
				\left. 
				||\nabla_z f(z)||^2 \leq \frac{\kappa_1}{\underline{\beta_2}\underline{\phi}\min\left\lbrace\frac12\kappa_1\underline{\beta_1}, \kappa_2\right\rbrace}
				\bigg|\bigg|
				\begin{bmatrix}
					d_{eq} \\
					d_{ineq}
				\end{bmatrix}
				\bigg|\bigg|^2
			\right\rbrace,
		\end{split}
	\end{equation}
	where $\underline{\phi} = \min_{\mathcal{A}\in\mathcal{S}}\phi_{\mathcal{A}}$, $\underline{\beta_1} = \min_{\mathcal{A}\in\mathcal{S}}\beta_{\mathcal{A}}^1$, and $\underline{\beta_2} = \min_{\mathcal{A}\in\mathcal{S}}\beta_{\mathcal{A}}^2$ are the minimum values over all possible sets $\mathcal{A}\in\mathcal{S}$.
	
	\item\label{prop:restrictedDynamics:positiveInvarianceOf_F0} Defining $\mathcal{F}0$ to be the smallest sublevel set of $f$ containing $\mathcal{F}$ and $z(0)$, the set $\mathcal{F}0$ is positively invariant.
	
	\item\label{prop:Existence:existenceUniqueness} There exists a unique, absolutely continuous, and bounded solution of the switched system on the interval $t\in\left[0,\infty\right)$.
	\end{enumerate}
\end{proposition}

\begin{proof}
	First, we verify Claim~\ref{prop:Existence:invariantSet}, which ensures the existence of a positively invariant compact set that contains the forward solution of the switched system. Noting that
	\begin{equation*}
		||d_\mathcal{A}||
		\leq
		\bigg|\bigg|
		\begin{bmatrix}
			d_{eq} \\
			d_{ineq}
		\end{bmatrix}
		\bigg|\bigg|,
	\end{equation*}
	for all $\mathcal{A}\in\mathcal{S}$, it follows that $\mathcal{F}_\mathcal{A}\subseteq\mathcal{F}$ for all $\mathcal{A}$. By Proposition~\ref{prop:subsystemBehaviour}, the solution to all subsystems is positively invariant on any sublevel set of $f$ containing $\mathcal{F}$, which implies that the solution to the switched system is similarly invariant on and sublevel set of $f$ containing $\mathcal{F}$.
	
	Claim~\ref{prop:restrictedDynamics:positiveInvarianceOf_F0} follows by noting that $\mathcal{F}\subseteq\mathcal{F}0$.
	
	To verify Claim \ref{prop:Existence:existenceUniqueness}, note that as the forward solution is positively invariant on $\mathcal{F}0$, each subsystem $\dot z = h_\mathcal{A}(z)$ is Lipschitz on the same domain. Combining this observation with the existence of an average dwell-time, we conclude from \cite[Theorem 1, Remark 2]{liyingExistenceUniquenessSolutions2012} that the solution to the switched system exists and is unique for all time.
\end{proof}

\subsection{Asymptotic behaviour}
So far, it has been established that the solution to the switched system exists, is unique, and is bounded. It has additionally been verified that the feasibility sets $\mathcal{G}_{eq}$, $\mathcal{G}_{ineq}$, $\Gamma$ are positively invariant. Our attention now turns to considering the asymptotic behaviour of the solution to ensure that it converges to a point satisfying the KKT conditions. 

We first restrict our analysis to only considering solutions contained within the feasible set $\Gamma$. The reason for restricting our analysis to this set is that the objective function $f$ can be verified as a common Lyapunov function for all subsystems on this set. We exploit this property to then characterise the limiting behaviour of the switched dynamics. Once the behaviour of the dynamics restricted to this set has been establish, continuity of the solutions is used to relate the restricted solutions to those of the full domain of interest.

The following proposition establishes that the solutions of the switched system restricted to the set $\Gamma$ converge to an invariant set that is characterised by the subsystem dynamics. When solutions are restricted to this set, the objective function $f$ qualifies as a common Lyapunov function and an invariance principle can be applied to obtain the desired result.
\begin{proposition}\label{prop:restrictedDynamics}
	Consider the switched system formed by combining the subsystem dynamics \eqref{optimisationDynamicsOriginal} with the switching law in Algorithm~\ref{alg:switching}. Recalling that $\mathcal{F}0$ is the smallest sublevel set of $f$ containing $\mathcal{F}$, defined in \eqref{FsetDefinitions2}, and $z(0)$, the following claims hold:
	\begin{enumerate}
		\item\label{prop:restrictedDynamics:commonLyapunov} On the set $\Gamma$, $f$ is a common Lyapunov function to all subsystems and satisfies
	 	\begin{equation}\label{fDerivativeRestricted}
	 		\dot f|_{\Gamma}
	 		=
	 		-\kappa_2 \nabla_z^\top f(z)G_\mathcal{A}^{\perp\top}(z)G_\mathcal{A}^\perp(z)\nabla_z f(z)
	 		\leq
	 		0
	 	\end{equation}
	 	for all $\mathcal{A}\in\mathcal{S}$. 
	 	\item\label{prop:restrictedDynamics:invariantSet} A solution of the switched system with $z(0)\in\Gamma$ converges to the largest weakly invariant set contained within
			\begin{equation}\label{equilibriumCondition}
				\begin{split}
					\Omega
					=
					\mathcal{F}0 \ \cap \ &\left\lbrace z\in\Gamma \ | \ \exists \mathcal{A}\in\mathcal{S} \ \text{with} \right.\\ &\left.G_{\mathcal{A}}^\perp(z)\nabla_z f(z) = 0_{(n-m-a)\times 1} \right\rbrace \subseteq\Gamma.
				\end{split}
			\end{equation}
	\end{enumerate}
\end{proposition}

\begin{proof}
	For Claim \ref{prop:restrictedDynamics:commonLyapunov}, to verify that $f$ is a common Lyapunov function for all subsystems on the set $\Gamma$, notice that the subsystem dynamics \eqref{optimisationDynamicsFactored} restricted to the set $\Gamma$, are given by
	\begin{equation}\label{optimisationDynamicsRestricted}
		\begin{split}
			\dot z|_{\Gamma}
			=
			h_\mathcal{A}|_{\Gamma}(z)
			=&
			-\kappa_2 G_\mathcal{A}^{\perp\top}(z)G_\mathcal{A}^\perp(z)\nabla_z f(z)
		\end{split}
	\end{equation}
	for all $\mathcal{A}\in\mathcal{S}$.
	The time derivative of $f$ evaluated along the subsystem dynamics \eqref{optimisationDynamicsRestricted} is given by \eqref{fDerivativeRestricted}. As this expression is non-positive for all subsystems, $f(z)$ is a common Lyapunov function.
	
	To prove Claim \ref{prop:restrictedDynamics:invariantSet}, we utilize the fact that $f$ is a common Lyapunov for solutions on $\Gamma$ and invoke invariance. The invariance principle for switched systems \cite[Theorem 3.6]{mancilla-aguilarInvariancePrinciplesSwitched2011} implies that the restricted dynamics converge to the largest weakly invariant set contained in \eqref{equilibriumCondition}. Note that \cite[Theorem 3.6]{mancilla-aguilarInvariancePrinciplesSwitched2011} requires the so-called $\mathbf{L}$ property, which is ensured by the average dwell-time condition verified in Proposition~\ref{prop:feasibleSetPosInvar}.\ref{prop:Existence:averageDwell} by \cite[Lemma 3.3]{mancilla-aguilarInvariancePrinciplesSwitched2011}. From Proposition~\ref{prop:feasibleSetPosInvar}.\ref{prop:restrictedDynamics:positiveInvariance}, the set $\Gamma$ is positively invariant, so $\Omega$ must be contained within $\Gamma$.
\end{proof}

It has now been established that solutions to the switched system starting in $\Gamma$ must asymptotically approach the set $\Omega$, defined in \eqref{equilibriumCondition}. We now begin the task of relating the elements of this set with the KKT conditions \eqref{KKTconditions}. As the set $\Omega$ is formed by considering the invariant sets of all possible subsystems, we start by inspecting the equilibrium points of each subsystem. The following proposition shows that for any equilibrium of any given subsystem, the corresponding active constraint vector must be equal to zero. Furthermore, there exist corresponding Lagrange multipliers that can be used to relate the equilibrium points to the KKT conditions.

\begin{proposition}\label{prop:equilibriumCharacterisation}
	A point $\bar z\in\mathbb{R}^n$ is an equilibrium of the subsystem \eqref{optimisationDynamicsFactored} if and only if it satisfies 
	\begin{equation}\label{equilibriumCharaterisation}
		\begin{split}
			g_\mathcal{A}(\bar z) &= 0_{(m+a)\times 1} \\
			G_\mathcal{A}^\perp(\bar z)\nabla_z f(\bar z) &= 0_{(n-m-a)\times 1}.
		\end{split}
	\end{equation}
	Furthermore, for any $\bar z\in\mathbb{R}^n$, there exist vectors $\bar\lambda\in\mathbb{R}^m$ and $\bar\nu\in\mathbb{R}^a$ such that
	\begin{equation}\label{gradientExpression}
		\nabla f(\bar z)
				+ \frac{\partial^\top g_{eq}}{\partial z}(\bar z)\bar\lambda
  + \frac{\partial^\top g_{ineq}}{\partial z}(\bar z)E^\top_\mathcal{A}\bar\nu
				= 0_{n\times 1}.
	\end{equation} 
\end{proposition}

\begin{proof}
		Let us first assume that there exists a point $\bar z$ satisfying \eqref{equilibriumCharaterisation}. Substituting such a point into \eqref{optimisationDynamicsFactored} verifies that $\bar z$ is an equilibrium. Now, considering the converse, assume that a point $\bar z$ is an equilibrium of the dynamics \eqref{optimisationDynamicsFactored}. Recalling that $G_\mathcal{A}^\perp(z)\frac{\partial^\top g_\mathcal{A}}{\partial z}(z) = 0_{(n-m-a)\times (m+a)}$ and left multiplying \eqref{optimisationDynamicsFactored} by $\frac{\partial g_\mathcal{A}}{\partial z}(\bar z)$ results in
		\begin{equation*}
			\begin{split}
				-\kappa_1 \underbrace{\frac{\partial g_\mathcal{A}}{\partial z}(\bar z)A_\mathcal{A}^{\top}(\bar z)}_{B_\mathcal{A}(\bar z)} B_\mathcal{A}^{-1}(\bar z)g_\mathcal{A}(\bar z)
				=
				-\kappa_1 g_\mathcal{A}(\bar z) 
				= 
				0_{(m+a)\times 1},
			\end{split}
		\end{equation*}
		which implies that $g_\mathcal{A}(\bar z) = 0_{(m+a)\times 1}$. Substituting this result back into \eqref{optimisationDynamicsFactored} and left multiplying by $\nabla_z^\top f(\bar z)$ results in
		\begin{equation*}
			\begin{split}
				-\kappa_2\nabla_z^\top f(\bar z)G_\mathcal{A}^{\perp\top}(\bar z)G_\mathcal{A}^\perp(\bar z)\nabla_z f(\bar z)
				&=
				0,
			\end{split}
		\end{equation*}
		which recovers the conditions \eqref{equilibriumCharaterisation}.

		To verify that there exist vectors $\bar\lambda, \bar\nu$ that satisfy \eqref{gradientExpression}, recall that as $G_\mathcal{A}^\perp$ is full rank left annihilator of $\frac{\partial^\perp g_\mathcal{A}}{\partial z}$. The implication of this fact is that the rows of $G_\mathcal{A}^\perp$ and $\frac{\partial^\top g_\mathcal{A}}{\partial z}$ span $\mathbb{R}^n$, which means that there exist vectors $\bar{\lambda}\in\mathbb{R}^m$, $\bar\nu\in\mathbb{R}^a$, and $\bar\mu\in\mathbb{R}^{n-m-a}$, such that
		\begin{equation}\label{df_decomposition}
			\nabla f(\bar z)
			=
			-
			\frac{\partial^\top g_\mathcal{A}}{\partial z}(\bar z)
			\begin{bmatrix}
				\bar\lambda \\
				\bar\nu
			\end{bmatrix}
			-
			G_\mathcal{A}^{\perp\top}(\bar z)
			\bar\mu.
		\end{equation}
		Left multiplying by $G_\mathcal{A}^{\perp}(\bar z)$ results in
		\begin{equation*}
			G_\mathcal{A}^{\perp}(\bar z)
			\nabla f(\bar z)
			=
			-
			G_\mathcal{A}^{\perp}(\bar z)
			G_\mathcal{A}^{\perp\top}(\bar z)
			\bar\mu
			=
			0_{(n-m-a)\times 1}
		\end{equation*}
		by \eqref{equilibriumCharaterisation}, which implies that $\bar\mu = 0_{(n-m-a)\times 1}$. Substituting this result into \eqref{df_decomposition} recovers \eqref{gradientExpression}. Using the definition of $g_\mathcal{A}(z)$ in \eqref{constraintsReduced}, it can be further verified that
		\begin{equation*}
			\begin{bmatrix}
				\bar\lambda \\
				\bar\nu
			\end{bmatrix}
			=
			B_\mathcal{A}^{-\top}(\bar z)
			d_\mathcal{A}.
		\end{equation*}
\end{proof}

Notice that the equilibrium condition for each subsystem \eqref{gradientExpression} is close to satisfying the first-order KKT conditions \eqref{KKTconditions}, but may not satisfy condition \eqref{KKTconditions:nuPositive}. That is, each element of the resulting $\bar\nu$ may not necessarily be positive. This issue is addressed in the following proposition by showing that subsystem equilibrium points that do not satisfy the KKT conditions cannot be invariant under the chosen switching law. Conversely, it is also shown that subsystem equilibrium points that do satisfy the KKT conditions are invariant under the switching law. Combining these two properties provides a characterisation of the attractive set $\Omega$.

\begin{proposition}\label{prop:LimitKKTconditions}
	The largest weakly invariant set contained within $\Omega$, defined in \eqref{equilibriumCondition}, contains only points that satisfy the KKT conditions \eqref{KKTconditions}.
\end{proposition}

\begin{proof}
	First, we establish that the set $\Omega$ can contain only equilibrium points of the subsystems. Note that on the set $\Gamma$, we have that $g_\mathcal{A}(z) = 0_{(m+a)\times 1}$ for all $z\in\Gamma$ and $\mathcal{A}\in\mathcal{S}$. Combining this property with the requirement that elements of $\Omega$ satisfy $G_{\mathcal{A}}^\perp(z)\nabla_z f(z) = 0_{(n-m-a)\times 1}$ for any given $\mathcal{A}$, the elements of $\Omega$ must be contained within the set of all equilibrium points of all subsystems by Proposition~\ref{prop:equilibriumCharacterisation}. Furthermore, for any subsystem equilibrium point $\bar z$ there exist corresponding vectors $\bar\lambda$ and $\bar\nu$ such that \eqref{gradientExpression} holds. We now examine which of these candidate points are invariant under the switching law in Algorithm~\ref{alg:switching}.
	
	Consider an arbitrary subsystem $\mathcal{A}\in\mathcal{S}$ at time $t_k$ with equilibrium point $\bar z\in\mathbb{R}^n$. We show that if $\sigma(t_k) = \mathcal{A} \subset \mathcal{I}(\bar z(t_k))$, then the switching law immediately updates such that $\sigma(t_k^+) = \mathcal{I}(\bar z)$. Furthermore, the point $\bar z$ remains an equilibrium point of the subsystem $\sigma(t_k^+)$. As the point $\bar z$ is an equilibrium, we have that $\frac{\partial g_{ineq,i}}{\partial z}(\bar z)h_{\sigma(t_k)}(\bar z) = 0$ for any $i\in\mathcal{I}(\bar z(t_k))\backslash\sigma(t_k)$. This scenario triggers the addition of the index $i$ to the active set by the switching law. The point remains an equilibrium after the switch as the conditions of Proposition~\ref{prop:equilibriumCharacterisation} continue to hold. Due to this fact, we will assume that $\sigma(t) = \mathcal{I}(\bar z(t))$ when considering equilibrium points of subsystems in subsequent arguments.
	
	We now consider the largest weakly invariant set contained in $\Omega$. Two properties are verified: 1)~If a point satisfies the KKT conditions, no further switching occurs, and it is invariant. 2)~If a point does not satisfy the KKT conditions, a switch occurs, and the point is not invariant.
	\begin{itemize}
	\item \textbf{Case 1 - $\bar z$, $\bar\lambda$, $\bar\nu$ satisfy the KKT conditions:} As we have that $\sigma(t) = \mathcal{I}(\bar z)$, it is not possible to add an index to the active set, so we only consider the potential impact of removing an index. As the point satisfies the KKT conditions, each element of $\bar\nu$ is non-negative. For removal of the index $k\in\sigma(t)$ from the active set, the switching law requires $\frac{\partial g_{ineq,k}}{\partial z}(\bar z)h_{\sigma(t)\backslash \left\lbrace k\right\rbrace}(\bar z) < 0$. Computing this expression results in
	\begin{equation}\label{switchConditionKKT}
		\begin{split}
			&\frac{\partial g_{ineq,k}}{\partial z}(\bar z)h_{\sigma(t)\backslash \left\lbrace k\right\rbrace}(\bar z)
			= \\
			&-\kappa_2\frac{\partial g_{ineq,k}}{\partial z}(\bar z)G_{\sigma(t)\backslash \left\lbrace k\right\rbrace}^{\perp\top}(\bar z)G_{\sigma(t)\backslash \left\lbrace k\right\rbrace}^\perp(\bar z)\nabla_z f(\bar z),
		\end{split}
	\end{equation}
	where we have substituted the subsystem dynamics from \eqref{optimisationDynamicsFactored} by noting that $g_{\sigma(t)} = 0_{(m+a)\times 1}$. As the tuple $\bar z$, $\bar\lambda$, $\bar\nu$ satisfy the KKT conditions, we can write
	\begin{equation*}
		\begin{split}
			\nabla f(\bar z)
			=
			- &\frac{\partial^\top g_{eq}}{\partial z}(\bar z)\bar\lambda
			- \frac{\partial^\top g_{ineq}}{\partial z}(\bar z) \bar\nu \\
			G_{\sigma(t)\backslash \left\lbrace k\right\rbrace}^\perp(\bar z)\nabla f(\bar z)
			&=
			- G_{\sigma(t)\backslash \left\lbrace k\right\rbrace}^\perp(\bar z)\frac{\partial^\top g_{ineq}}{\partial z}(\bar z) \bar\nu \\
			G_{\sigma(t)\backslash \left\lbrace k\right\rbrace}^\perp(\bar z)\nabla f(\bar z)
			&=
			- G_{\sigma(t)\backslash \left\lbrace k\right\rbrace}^\perp(\bar z)\frac{\partial^\top g_{ineq,k}}{\partial z}(\bar z) \bar\nu_k 
		\end{split}
	\end{equation*}
	due to the fact that it is a left annihilator for all constraints except for the $k^{th}$ inequality constraint. Substituting this expression into \eqref{switchConditionKKT} results in
	\begin{equation*}
		\begin{split}
			&\frac{\partial g_{ineq,k}}{\partial z}(\bar z)h_{\sigma(t)\backslash \left\lbrace k\right\rbrace}(\bar z)
			= \\
			&\bar\nu_k\kappa_2\frac{\partial g_{ineq,k}}{\partial z}(\bar z)G_{\sigma(t)\backslash \left\lbrace k\right\rbrace}^{\perp\top}(\bar z)G_{\sigma(t)\backslash \left\lbrace k\right\rbrace}^\perp(\bar z)\frac{\partial^\top g_{ineq,k}}{\partial z}(\bar z).
		\end{split}
	\end{equation*}
	As each element of $\bar\nu$ is non-negative, this expression is also non-negative. Consequently, a constraint index cannot be removed from the active set if the equilibrium point satisfies the KKT conditions and no further switching occurs.
	
	\item \textbf{Case 2 - $\bar z$, $\bar\lambda$, $\bar\nu$ do not satisfy the KKT conditions:}
	As the tuple $\bar z$, $\bar\lambda$, $\bar\nu$ do not satisfy the KKT conditions, at least one of the element of $\bar\nu$ must be negative. Assume that the $k^{th}$ element is negative. Left multiplying the expression \eqref{gradientExpression} by $G^\perp_{\sigma(t)}(\bar z)$ results in
	\begin{equation*}
		\begin{split}
			G^\perp_{\sigma(t)\backslash k}(\bar z)\nabla f(\bar z)
			  &=
			  -G^\perp_{\sigma(t)\backslash k}(\bar z)\frac{\partial^\top g_{ineq}}{\partial z}(\bar z)E^\top_{\sigma(t)}\bar\nu \\
			  &=
			  -G^\perp_{\sigma(t)\backslash k}(\bar z)\frac{\partial^\top g_{ineq,k}}{\partial z}(\bar z)\bar\nu_k. 
		\end{split}
	\end{equation*}
	Substituting this relationship into the subsystem dynamics of $\mathcal{A}\backslash k$ in \eqref{optimisationDynamicsFactored} results in
	\begin{equation*}
		h_{\sigma(t)\backslash k}(\bar z)
		=
		\kappa_2\bar\nu_k G_{\sigma(t)\backslash k}^{\perp\top}(z)G^\perp_{\sigma(t)\backslash k}(\bar z)\frac{\partial^\top g_{ineq,k}}{\partial z}(\bar z).
	\end{equation*}
	Considering the gradient of the $k^{th}$ inequality constraint along this flow vector results in
	\begin{equation*}
		\begin{split}
			&h_{\mathcal{A}\backslash k}(\bar z)\frac{\partial g_{ineq,k}}{\partial z}(\bar z)
			= \\
			&\kappa_2\bar\nu_k \frac{\partial g_{ineq,k}}{\partial z}(\bar z)G_{\mathcal{A}\backslash k}^{\perp\top}(z)G^\perp_{\mathcal{A}\backslash k}(\bar z)\frac{\partial^\top g_{ineq,k}}{\partial z}(\bar z) < 0,
		\end{split}
	\end{equation*}
	which results in the removal of the index $k$ from the active set by the switching law in Algorithm~\ref{alg:switching}. After the switch, the time derivative of $f$ is given by
	\begin{equation*}
		\begin{split}
			\dot f
			=&
			-\kappa_2
			\nabla_z^\top f(\bar z)
			G_{\sigma(t)\backslash k}^{\perp\top}(z)G^\perp_{\sigma(t)\backslash k}(\bar z)
			\nabla_z f(\bar z) \\
			=&
			-\kappa_2(\bar\nu_k)^2
			\frac{\partial g_{ineq,k}}{\partial z}(\bar z)
			G_{\sigma(t)\backslash k}^{\perp\top}(z) \\
			&\times G^\perp_{\sigma(t)\backslash k}(\bar z)
			\frac{\partial^\top g_{ineq,k}}{\partial z}(\bar z) < 0.
		\end{split}
	\end{equation*}
	Therefore, if a point does not satisfy the KKT conditions, a switch occurs, which results in a strict decrease of the objective function $f$, ensuring that the point is not invariant.
	\end{itemize}
 This concludes the proof.
\end{proof}

\subsection{Proof of Theorem~\ref{thm:convergence}}\label{sec:mainResultProof}
%

We now present the proof of the main result, namely that the solution of the switched system obtained by combining the subsystem dynamics~\eqref{optimisationDynamicsOriginal} with the switching law in Algorithm~\ref{alg:switching} converges to a point $z^\star \in \mathbb{R}^n$ satisfying the first-order KKT conditions~\eqref{KKTconditions}. Propositions~\ref{prop:restrictedDynamics} and~\ref{prop:LimitKKTconditions} establish convergence to a KKT point for solutions whose initial conditions satisfy the constraints, and it therefore remains to show that solutions with initial conditions outside the equality constraint set also converge to such points. This follows from the uniqueness, continuity, and boundedness of solutions: since the equality constraints converge uniformly at an exponential rate, trajectories of the switched system approach the set $\Gamma$, and continuity implies that all solutions with initial conditions in $\mathcal{G}_{ineq}$ share the same limit set $\Omega$.

\begin{proofof}[Proof of Theorem~\ref{thm:convergence}]
	The proof follows from combining the previous claims. From Proposition~\ref{prop:existence_uniqueness}.\ref{prop:Existence:existenceUniqueness}, the solution of the switched system exists and is unique for all time. It is additionally positively invariant on the compact set $\mathcal{F}0$. Following the proof of \cite[Lemma 4.1]{khalilNonlinearSystems2001}, which requires continuity and boundedness of the solutions, the solution tends to a non-empty, compact, and invariant limit set $L$. Proposition~\ref{prop:feasibleSetPosInvar}.\ref{prop:subsystemBehaviour:gA} ensures that $g_{eq}(t)$ converges at an exponential rate for all $\sigma(t)$, and Proposition~\ref{prop:feasibleSetPosInvar}.\ref{prop:restrictedDynamics:positiveInvariance} ensures that $\mathcal{G}_{ineq}$ is positively invariant, which implies that the limit set $L$ is contained in $\Gamma$. By Proposition~\ref{prop:restrictedDynamics}.\ref{prop:restrictedDynamics:invariantSet}, all solutions within the set $\Gamma$ converge to $\Omega$, implying that $L\subseteq\Omega$. Finally, Proposition~\ref{prop:LimitKKTconditions} verifies that all points within the largest invariant set contained in $\Omega$ satisfy the KKT conditions. Therefore, we conclude that solutions of the switched system with $z(0)\in\mathcal{G}_{ineq}$ converge to a point satisfying the KKT conditions.
\end{proofof}

\begin{remark}\label{remark:tuning}
    In this paper, we have verified asymptotic convergence of the optimisation dynamics to a point satisfying the KKT conditions, but have not examined the rate of convergence. Despite this, we can provide some commentary on how the tuning parameters $\kappa_1, \kappa_2$ will influence the convergence rate. The constraint equations converge at the exponential rate \eqref{constrait_dynamics}, therefore, larger $\kappa_1$ will result in faster rate of convergence to the feasible set $\Gamma$. On the set $\Gamma$ the objective function decays according to \eqref{fDerivativeRestricted}, which is accelerated by increasing $\kappa_2$. Therefore, we conclude that increasing the value of $\kappa_1, \kappa_2$ will increase the rate of convergence to a solution.
\end{remark}


\section{Applications}\label{sec:examples}
In this section, three applications of the proposed optimisation dynamics are presented. In each case, it is shown how the constraints can be formulated in the form \eqref{probFormulation} and the technical assumptions are verified. The code used to generate the plots is available at \url{https://github.com/JoelFerguson/Switching_Dynamics_That_Converges_To_The_KKT_Point_Of_A_Nonlinear_Optimization_Problem}.

\subsection{Quadratic programming problem}
In this section, we show how general quadratic programming problems with positive definite objective functions can be solved using the proposed optimisation dynamics.

\subsubsection{Constraint formulation}
Consider the following quadratic programming problem with positive definite quadratic cost and linear constraints:
\begin{equation*}
	\begin{split}
		\underset{z}{\mathrm{argmin}} \ \frac12 z^\top Lz + K^\top z \ \ \ 
		s.t. \ &B_{eq}z + c_{eq} = 0_{m\times 1} \\
		&B_{ineq}z + c_{ineq} \leq 0_{p\times 1},
	\end{split}
\end{equation*}
where $L\in\mathbb{R}^{n\times n}$ is symmetric positive definite, $K\in\mathbb{R}^{n}$, $B_{eq}\in\mathbb{R}^{m\times n}$, $B_{ineq}\in\mathbb{R}^{p\times n}$, $c_{eq}\in\mathbb{R}^{m\times 1}$, and $c_{ineq}\in\mathbb{R}^{p\times 1}$. Noting that $f(z) = \frac12 z^\top Lz + K^\top z$ is the objective function, the problem can be recast into the form \eqref{probFormulation} by modifying the constraint equations as
\begin{equation*}
	\begin{split}
		\underbrace{B_{eq}L^{-1}}_{A_{eq}}\nabla f + \underbrace{\left(c_{eq} - B_{eq}L^{-1}K\right)}_{d_{eq}} &= 0_{m\times 1} \\
		\underbrace{B_{ineq}L^{-1}}_{A_{ineq}}\nabla f + \underbrace{\left(c_{ineq} - B_{ineq}L^{-1}K\right)}_{d_{ineq}} &\leq 0_{p\times 1}.
	\end{split}
\end{equation*}
To verify Assumption \ref{Assumption1}, we construct $\mathcal{B}_\mathcal{A}$ as
\begin{equation*}
	\mathcal{B}_\mathcal{A}
	=
	\frac{\partial g_\mathcal{A}}{\partial z}A_\mathcal{A}^\top
	=
	\underbrace{
	\begin{bmatrix}
		B_{eq} \\ E_\mathcal{A}B_{ineq}
	\end{bmatrix}}_{\frac{\partial g_\mathcal{A}}{\partial z}}
	\underbrace{
	L^{-1}
	\begin{bmatrix}
		B_{eq} \\ E_\mathcal{A}B_{ineq}
	\end{bmatrix}^\top}_{A_\mathcal{A}^\top}.
\end{equation*}
This expression is constant and positive definite for any $\mathcal{A}\in\mathcal{S}$. Therefore, Assumption \ref{Assumption1} holds for quadratic programming problems with positive definite $L$. As the second derivative of the associated Lagrangian function with respect to $z$ is positive definite, so any solution is locally optimal by Theorem~\ref{thm:KKTsufficient}.

\subsubsection{Numerical example}
The problem was tested for
\begin{multicols}{2}
	\noindent
	\begin{equation*}
		\begin{split}
			L\!=\! 
			\begin{bmatrix}
				1 & -1 \\ -1 & 2
			\end{bmatrix}\!,
			B_{ineq}\! =\! 
			\begin{bmatrix}
				1 & 1 \\
				-1 & 2
			\end{bmatrix}\!,
			K \!=\! 
			\begin{bmatrix}
				-2 \\ -6
			\end{bmatrix}\!,
			C_{ineq} \!=\! 
			\begin{bmatrix}
				-2 \\
				-2
			\end{bmatrix}\!.
		\end{split}
	\end{equation*}
\end{multicols}
The tuning parameters for the optimisation dynamics were set to $\kappa_1 = 1$, $\kappa_2 = 1$, $\delta T = 0.1$, and the dynamics was initialised at $z(0) = \left[-0.25, 0\right]$.

The dynamic response of the optimisation dynamics is shown in Figure~\ref{qpExample}. The red and yellow lines show the boundaries of the inequality constraints, whereas the green line shows the dynamic response of the optimisation dynamics. Initially, the dynamics are unconstrained, and the dynamics follow a path of gradient descent. The solution then intersects the second inequality constraint, and the dynamics move along the constraint surface, seeking a minimum for the objective function. The solution eventually intersects the first inequality constraint, which fully constrains the solution. The final solution of the dynamics was at the point $z^\star = \left[0.667, 1.33\right]$.
\begin{figure}[ht!]
	\centering{}
	\includegraphics[width=0.90\columnwidth]{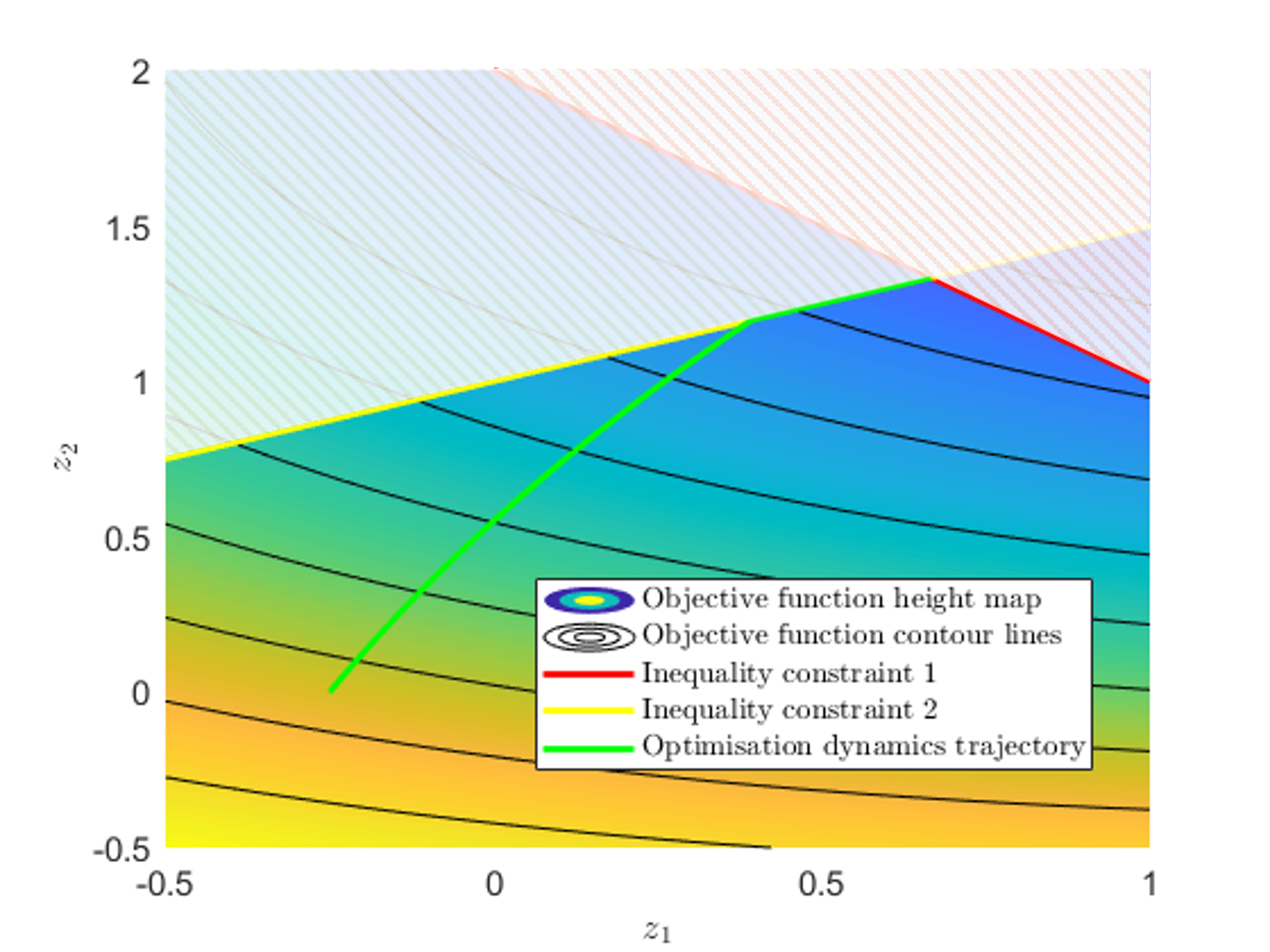}
	\caption{The trajectory of the optimisation dynamics seeking the minimiser of a QP problem subject to linear inequality constraints.}
	\label{qpExample}
\end{figure}

\subsection{Rosenbrock's function optimisation}
Rosenbrock's function is a commonly used benchmark for testing the performance of nonconvex optimisation algorithms as it is known to be difficult to minimise. Here, we consider the problem of minimising Rosenbrock's function
\begin{equation*}
	\begin{split}
		f(z)
		&=
		100(z_2 - z_1^2)^2 + (1 - z_1)^2
	\end{split}
\end{equation*}
subject to the linear inequality constraint
\begin{equation}\label{eg:rosenbrock:gineq}
	\begin{split}
		\underbrace{-2z_1 + z_2 + \frac{3}{4}}_{g_{ineq(z)}} \leq 0.
	\end{split}
\end{equation}
The representation of the inequality constraint is not currently in a gradient form as per \eqref{probFormulation}. To transform it into the form~\eqref{probFormulation}, we first evaluate the gradient of the objective function
\begin{equation*}
	\begin{split}
		\nabla_z f
		&=
		\begin{bmatrix}
			400z_1^2+2 & -400z_1 \\
			-200z_1 & 200
		\end{bmatrix}
		\begin{bmatrix}
			z_1 \\
			z_2
		\end{bmatrix}
		+
		\begin{bmatrix}
			-2 \\ 0	
		\end{bmatrix},
	\end{split}
\end{equation*}
which can be rearranged to find the relationship:
\begin{equation*}
	\begin{split}
		\begin{bmatrix}
			z_1 \\ z_2
		\end{bmatrix}
		&=
		\begin{bmatrix}
			\frac12 & z_1 \\
			\frac12 z_1 & z_1^2 + \frac{1}{200}
		\end{bmatrix}
		\nabla_z f
		+
		\begin{bmatrix}
			1 \\ z_1
		\end{bmatrix}.
	\end{split}
\end{equation*}
Expanding this expression and substituting the first line into the second results in
\begin{equation*}
	\begin{split}
		\begin{bmatrix}
			z_1 \\ z_2
		\end{bmatrix}
		&=
		\begin{bmatrix}
			\frac12\nabla_{z_1} f + z_1\nabla_{z_2} f + 1 \\
			(\frac12 z_1 + \frac12)\nabla_{z_1} f + (z_1^2 + z_1 + \frac{1}{200})\nabla_{z_2} f + 1
		\end{bmatrix},
	\end{split}
\end{equation*}
allowing the inequality constraint to be written as
\begin{equation*}
	\begin{split}
		\underbrace{
		\begin{bmatrix}
			-\frac{1}{2}+\frac{1}{2}z_1 & z_1^2-z_1+\frac{1}{200}
		\end{bmatrix}
		}_{A_{ineq}(z)}
		\nabla_z f
		+
		\underbrace{
		-\frac{1}{4}
		}_{d_{ineq}}
		\leq
		0.
	\end{split}
\end{equation*}
We now need to verify Assumption \ref{Assumption1} before proceeding. Note that the constraint~\eqref{eg:rosenbrock:gineq} has gradient
\begin{equation*}
	\begin{split}
		\frac{\partial g_{ineq}}{\partial z}
		=
		\begin{bmatrix}
			-2 & 1
		\end{bmatrix}.
	\end{split}
\end{equation*}
When the inequality constraint is inactive the matrix $B_\mathcal{A}$ is empty, and when the constraint is active it is given by
\begin{equation*}
	\begin{split}
		B_\mathcal{A}(z)
		&=
		\frac{\partial g_{ineq}}{\partial z}(z)A_\mathcal{A}^\top(z) \\
		&=
		\begin{bmatrix}
			-2 & 1
		\end{bmatrix}
		\begin{bmatrix}
			-\frac{1}{2}+\frac{1}{2}z_1 \\
			z_1^2-z_1+\frac{1}{200}
		\end{bmatrix}
		=
		z_1^2 - 2z_1 + \frac{201}{200}.
	\end{split}
\end{equation*}
This polynomial is positive-definite for all $z_1$, verifying Assumption~\ref{Assumption1}.

The optimisation dynamics were applied to Rosenbrock's function subject to the linear inequality constraint~\eqref{eg:rosenbrock:gineq}. The results are shown in Figure \ref{RosenbrockEg}. The objective function is indicated by the coloured height map, and the boundary of the inequality is indicated by the red line. {The tuning parameters were set to $\kappa_1 = 1$, $\kappa_2 = 1$, $\delta T = 0.1$ and} dynamics were initialised at $z(0) = \left[1.0, -1.0\right]$, which satisfies the inequality constraint. The dynamics follow a gradient descent until the solution is such that the inequality constraint holds with equality. At this point, the solution slides along the boundary of the constraint set until the objective function gradient points away from the constraint. The solution then moves away from the constraint and finds the function minimiser at $z^\star = \left[1.0, 1.0\right]$.

\begin{figure}[ht!]
	\centering{}
	\includegraphics[width=0.90\columnwidth]{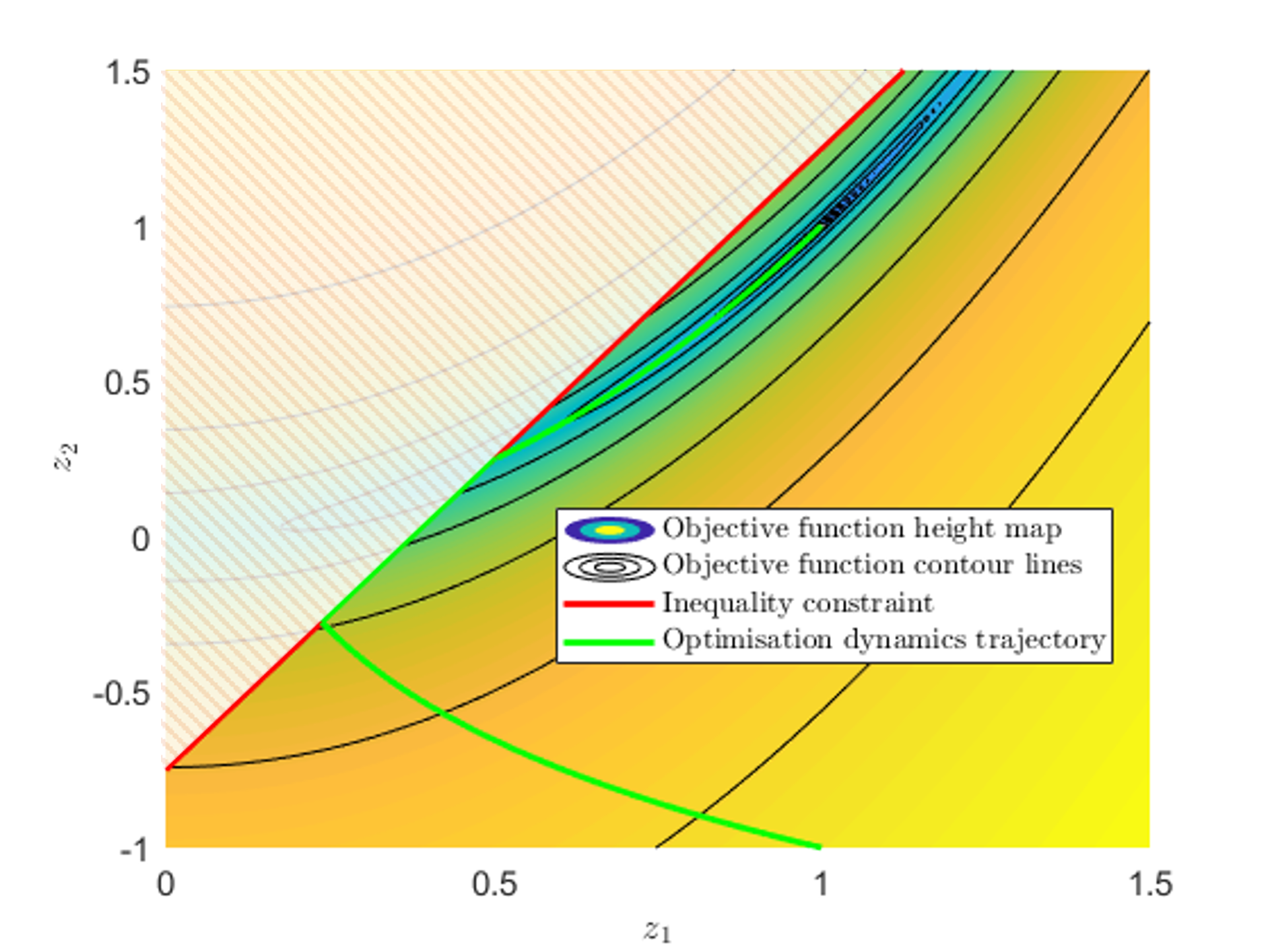}
	\caption{The trajectory of the optimisation dynamics seeking the minimiser of Rosenbrock's function.}
	\label{RosenbrockEg}
\end{figure}

\subsection{Energy-Efficient Buildings}\label{sec:examples:heatSharing}
In this section, we demonstrate how the proposed optimisation dynamics can be used to find a solution to a nonconvex optimisation problem aimed at enhancing the comfort of the occupants while minimising the power consumption of energy-efficient buildings with multiple zones.

\subsubsection{Constraint formulation}
The thermal dynamics of a building with multiple zones can be described by (see \cite{hatanakaIntegratedDesignOptimization2017}):
\begin{equation}\label{heatingDynamics}
	\begin{split}
		\underbrace{\begin{bmatrix}
			C_1 & 0 \\
			0 & C_2
		\end{bmatrix}}_{C}u
		&
		\begin{bmatrix}
			\dot T_1 \\
			\dot T_2
		\end{bmatrix}
		=
		-
		\underbrace{\begin{bmatrix}
			E_{11} & E_{12} \\
			E_{12}^\top & E_{22}
		\end{bmatrix}}_{E}
		\begin{bmatrix}
			T_1 \\
			T_2
		\end{bmatrix} \\
		&+
		\begin{bmatrix}
			R_1^{-1}(T^a - T_1) + A\operatorname{diag}(T^s - T_1)m + q \\
			R_2^{-1}(T^a - T_2)
		\end{bmatrix},
	\end{split}
\end{equation}
where $T_1\in\mathbb{R}^{n_1}$ are the temperatures in rooms equipped with heating, ventilation, and air-conditioning (HVAC) systems, whereas $T_2\in\mathbb{R}^{n_2}$ are the temperatures in rooms with no temperature control. The matrix $C$ describes the thermal capacitance of each room, whereas $E$ is a symmetric positive semi-definite matrix that describes the thermal resistance between each adjacent room. $T^a$ is a vector containing the ambient air temperature, and $R_i^{-1}$ describes the thermal resistance between each room and the ambient air. The matrix $A$ is diagonal, with each entry containing the specific heat of the air entering each zone. $T^s$ is the vector of supply air temperatures into each controlled zone, and $m$ is the control input describing the mass flow rate into each zone. The vector $q$ describes any constant thermal loads in each zone.

A common optimisation problem related to thermal dynamics is to find an equilibrium of the dynamic system \eqref{heatingDynamics} that minimises some objective function. The equilibrium condition is enforced by constructing equality constraints that are precisely the dynamics \eqref{heatingDynamics} with $\dot T_1 = 0_{n_1\times 1}$ and $\dot T_2 = 0_{n_2\times 1}$. Due to cross-terms between the temperature $T_1$ and the input $m$, the dynamics~\eqref{heatingDynamics} is bilinear and the corresponding optimisation problem becomes nonconvex. A possible cost function that is meaningful from an engineering viewpoint is given by
\begin{equation}
\label{eq:cost_func}
	f(z)
	=
	\frac12||T_1 - T_1^\star||^2_{L_1} + \frac12||T_2 - T_2^\star||^2_{L_2} + \frac12||m||^2_{L_m},
\end{equation}
where $T_1^\star$ and $T_2^\star$ are the vectors of desired temperatures for each zone, $L_1,\ L_2$, and $L_m$ are positive-definite and diagonal weighting matrices for each term in the cost function, and $z = \left[T_1^\top, T_2^\top, m^\top\right]$ is the vector of decision variables. The first two terms in the cost function~\eqref{eq:cost_func} ensure thermal comfort for the occupants of the building and the last term ensures minimisation of power consumption. The gradient of the cost function is given by
\begin{equation*}
	\nabla f
	=
	\begin{bmatrix}
		L_1(T_1 - T_1^\star) \\
		L_2(T_2 - T_2^\star) \\
		L_m m
	\end{bmatrix}.
\end{equation*}

To ensure that any solution to the optimisation problem is a feasible equilibrium, the solution should satisfy~\eqref{heatingDynamics} subject to $\dot T_1 = 0_{n_1\times 1}$ and $\dot T_2 = 0_{n_2\times 1}$. The corresponding constraint equation can be expressed in the form \eqref{probFormulation} with
\small
\begin{equation*}
	\begin{split}
		A_{eq}(z)
		=&
		\begin{bmatrix}
			-E_{11}-R_1^{-1}-A\operatorname{diag}(m) & -E_{12} \\
			-E_{12}^\top & -E_{22}-R_2^{-1} \\
			A\operatorname{diag}(T^s - T_1^\star) & 0
		\end{bmatrix}^\top \\
		&\times
		\begin{bmatrix}
			L_1^{-1} & 0 & 0 \\
			0 & L_2^{-1} & 0 \\
			0 & 0 & L_m^{-1}
		\end{bmatrix} \\
		d_{eq}
		=&
		\begin{bmatrix}
			R_1^{-1}T^a + q \\
			R_2^{-1}T^a
		\end{bmatrix}
		-
		\begin{bmatrix}
			E_{11}+R_1^{-1} & E_{12} \\
			E_{12}^\top & E_{22}+R_2^{-1}
		\end{bmatrix}
		\begin{bmatrix}
			T_1^\star \\
			T_2^\star
		\end{bmatrix}.
	\end{split}
\end{equation*}
\normalsize
We note that at any equilibrium, we additionally require that the room temperatures $T_1$ must be strictly less than the supply temperature $T^s$. That is, $T^s-T_1 \geq \epsilon$ must hold for some $\epsilon > 0$. The inequality constraint can be written in the form \eqref{probFormulation} as
\begin{equation*}
	\begin{split}
		A_{ineq}(z)
		&=
		\begin{bmatrix}
			L_1^{-1} & 0 & 0
		\end{bmatrix} \\
		d_{ineq}
		&=
		\begin{bmatrix}
			T_1^\star - T^s + \epsilon
		\end{bmatrix}
		.
	\end{split}
\end{equation*}

Verification of Assumption~\ref{Assumption1} requires investigation of the constraint equations.
Combining the equality and inequality constraints, the full constraint  vector can be written in the form of \eqref{constraintsReduced} as
\small
\begin{equation*}
	\begin{split}
		&g_\mathcal{A}(z)
		= \\
		&\begin{bmatrix}
			-E_{11}-R_1^{-1}-A\operatorname{diag}(m) & -E_{12} & A\operatorname{diag}(T^s - T_1^\star) \\
			-E_{12}^\top & -E_{22}-R_2^{-1} & 0 \\
			E_\mathcal{A} & 0 & 0
		\end{bmatrix} \\
		&\times
		\begin{bmatrix}
			T_1-T_1^\star \\
			T_2-T_2^\star \\
			m
		\end{bmatrix}
		+
		\begin{bmatrix}
			-E_{11}T_1^\star - E_{12}T_2^\star + R_1^{-1}(T^a - T_1^\star) + q \\
			-E_{12}^\top T_1^\star - E_{22}T_2^\star + R_2^{-1}(T^a - T_2^\star) \\
			E_\mathcal{A}(T_1^\star - T^s)
		\end{bmatrix}.
	\end{split}
\end{equation*}
\normalsize
which has the Jacobian
\small
\begin{equation}\label{grad_ga}
	\begin{split}
		&\frac{\partial g_\mathcal{A}}{\partial z}(z)
		= \\
		&
		\begin{bmatrix}
			-E_{11}-R_1^{-1}-A\operatorname{diag}(m) & -E_{12} & A\operatorname{diag}(T^s - T_1) \\
			-E_{12}^\top & -E_{22}-R_2^{-1} & 0 \\
			E_\mathcal{A} & 0 & 0
		\end{bmatrix}.
	\end{split}
\end{equation}
\normalsize
Now we construct the term $B_\mathcal{A}(z)$, defined in \eqref{Bdefinition}, which results in the expression
\small
\begin{equation*}
	\begin{split}
		&B_\mathcal{A}(z)
		=
		\frac{\partial g_\mathcal{A}}{\partial z}(z)A_\mathcal{A}^\top(z) \\
		&=
		\begin{bmatrix}
			-E_{11}-R_1^{-1}-A\operatorname{diag}(m) & -E_{12} \\
			-E_{12}^\top & -E_{22}-R_2^{-1} \\
			E_\mathcal{A} & 0
		\end{bmatrix} \\
		&\times
		\begin{bmatrix}
			L_1^{-1} & 0 \\
			0 & L_2^{-1}
		\end{bmatrix}
		\begin{bmatrix}
			-E_{11}-R_1^{-1}-A\operatorname{diag}(m) & -E_{12} \\
			-E_{12}^\top & -E_{22}-R_2^{-1} \\
			E_\mathcal{A} & 0
		\end{bmatrix}^\top \\
		&+
		\begin{bmatrix}
			A\operatorname{diag}(T^s - T_1^\star)L_m^{-1}\operatorname{diag}(T^s - T_1)A & 0 & 0 \\
			0 & 0 & 0 \\
			0 & 0 & 0
		\end{bmatrix}.
	\end{split}
\end{equation*}
\normalsize
This expression is positive definite for all $T^s - T_1 \geq \epsilon$, which is ensured by the inequality constraint. We conclude that Assumption~\ref{Assumption1} is satisfied.

\subsubsection{Sufficient condition for local optimality}
Satisfaction of Assumption~\ref{Assumption1} ensures that the optimisation dynamics will converge to a point satisfying the KKT conditions by Theorem~\ref{thm:convergence}. Further investigation, however, is required to determine if such a point is locally optimal. A second order sufficient condition for local optimality was given in Theorem~\ref{thm:KKTsufficient}, which inspects the sign of the Hessian of the Lagrangian in directions consistent with the constraints. Following the discussion in \cite[Section 12.4]{nocedalNumericalOptimization1999}, this condition can be equivalently stated in terms of the projected Hessian. Using the definition of $G_\mathcal{A}^\perp(z)$, defined in \eqref{gPerpDef}, a triple $(x^\star,\lambda^\star,\mu^\star)$ is a locally optimal solution to \eqref{probFormulation} if it satisfies the KKT conditions and 
	\begin{equation}\label{projectedHessian}
		G^\perp_\mathcal{A}(z^\star)\frac{\partial^2\mathcal{L}}{\partial z^2}(z^\star,\lambda^\star,\nu^\star)G^{\perp\top}_\mathcal{A}(z^\star)
		>
		0.
\end{equation}

The Lagrangian for the building heating network problem can be written as
\begin{equation*}
	\begin{split}
		\mathcal{L}
		&=
		f(x)
		+
		\begin{bmatrix}
			\lambda_1^\top & \lambda_2^\top
		\end{bmatrix}
		g_{eq}(z)
		+
		\nu^\top g_{ineq}(z),
	\end{split}
\end{equation*}
where $g_{eq}$ and $g_{ineq}$ are defined as per \eqref{probFormulation}.
We can now compute the Hessian of the Lagrangian with respect to $z$ to assess the sufficient condition for local optimality. As all of the inequality constraints are linear, they do not impact the Hessian. Calculations reveal the Hessian as
\begin{equation}\label{heatingHessian}
	\begin{split}
		\frac{\partial^2\mathcal{L}}{\partial z^2}
		=
		\begin{bmatrix}
			L_1 & 0 & -A\operatorname{diag}(\lambda_1) \\
			0 & L_2 & 0 \\
			-A\operatorname{diag}(\lambda_1) & 0 & L_m
		\end{bmatrix}.
	\end{split}
\end{equation}
The presence of the Lagrange multipliers $\lambda_1$ in this expression makes the sign indefinite. To determine an expression for the Lagrange multiplier, consider the derivative 
\begin{equation*}
	\begin{split}
		\frac{\partial\mathcal{L}}{\partial m}
		=
		L_m m + A\operatorname{diag}(T^s - T_1)\lambda_1,
	\end{split}
\end{equation*}
which must be equal to zero at an equilibrium by the KKT condition \eqref{KKTconditions:gradient}. As $T^s - T_1 \geq \epsilon$ via the inequality constraints, we have that
\begin{equation*}
	\lambda_1
	=
	-\operatorname{diag}(T^s - T_1)^{-1}A^{-1}L_m m.
\end{equation*}
This Lagrange multiplier can be substituted into \eqref{heatingHessian} to find the Hessian of the Lagrangian as 
\begin{equation*}
	\begin{split}
		\frac{\partial^2\mathcal{L}}{\partial z^2}
		=
		\begin{bmatrix}
			L_1 & 0 & X(T_1,m) \\
			0 & L_2 & 0 \\
			X(T_1,m) & 0 & L_m
		\end{bmatrix},
	\end{split}
\end{equation*}
where $X(T_1,m) = \operatorname{diag}(T^s - T_1)^{-1}L_m\operatorname{diag}(m)$. Unfortunately, this expression is still sign indefinite, so we will instead consider the projected Hessian introduced in \eqref{projectedHessian}. We let $E_\mathcal{A}^\perp$ be a full rank left annihilator of $E_\mathcal{A}^\top$, satisfying $E_\mathcal{A}^\perp E_\mathcal{A}^\top = 0$. The left annihilator of the constant gradient \eqref{grad_ga} is given by
\begin{equation*}
	\begin{split}
		G_\mathcal{A}^\perp(z)
		&=
		E_\mathcal{A}^\perp
		\underbrace{
		\begin{bmatrix}
			I \\ 
			-(E_{22} + R_2^{-1})^{-1}E_{12}^\top \\
			\operatorname{diag}(T^s-T_1)^{-1}A^{-1}\left\lbrace Y+A\operatorname{diag}(m)\right\rbrace
		\end{bmatrix}^\top}_{:= Z} \\
		Y &= E_{11}+R_1^{-1}-E_{12}(E_{22}+R_2^{-1})E_{12}^\top.
	\end{split}
\end{equation*}
Using this expression, the projected Hessian \eqref{projectedHessian} can be computed as
\begin{equation*}
	\begin{split}
		G^\perp_\mathcal{A}\frac{\partial^2\mathcal{L}}{\partial z^2}G^{\perp\top}_\mathcal{A}
		=
		E_\mathcal{A}^\perp Z\frac{\partial^2\mathcal{L}}{\partial z^2}Z^\top E_\mathcal{A}^{\perp\top},
	\end{split}
\end{equation*}
which is positive for any combination of inequality constraints if $Z\frac{\partial^2\mathcal{L}}{\partial z^2}Z^\top > 0$. Expanding this expression yields
\begin{equation}\label{projectedHessianComputations1}
	\begin{split}
		&Z\frac{\partial^2\mathcal{L}}{\partial z^2}Z^\top \\
		&=
		L_1
		+
		E_{12}(E_{22} + R_2^{-1})^{-1}L_2(E_{22} + R_2^{-1})^{-1}E_{12}^\top \\
		&+
		\left\lbrace Y+A\operatorname{diag}(m)\right\rbrace
		\operatorname{diag}(T^s-T_1)^{-2}A^{-2}
		\left\lbrace Y+A\operatorname{diag}(m)\right\rbrace \\
		&+
		\left\lbrace Y+A\operatorname{diag}(m)\right\rbrace A^{-1}\operatorname{diag}(T^s-T_1)^{-1}X \\
		&+
		X\operatorname{diag}(T^s-T_1)^{-1}A^{-1}\left\lbrace Y+A\operatorname{diag}(m)\right\rbrace.
	\end{split}
\end{equation}
Considering the last two terms of this expression, we can substitute in for the expression $X$ and complete the squares,
\begin{equation}\label{projectedHessianComputations2}
	\begin{split}
		&\left\lbrace Y+A\operatorname{diag}(m)\right\rbrace A^{-1}\operatorname{diag}(T^s-T_1)^{-2}L_m\operatorname{diag}(m) \\
		&+
		L_m\operatorname{diag}(m)\operatorname{diag}(T^s-T_1)^{-2}A^{-1}\left\lbrace Y+A\operatorname{diag}(m)\right\rbrace \\
		&=
		2\left[
			\operatorname{diag}(m) + \frac12 YA^{-1}
		\right]
		L_m \operatorname{diag}(T^s-T_1)^{-2} \\
		&\times
		\left[
			\operatorname{diag}(m) + \frac12 YA^{-1}
		\right]^\top \\
		&-
		\frac12 YA^{-1}L_m \operatorname{diag}(T^s-T_1)^{-2}A^{-1}Y.
	\end{split}
\end{equation}
As the inequality constraint $T^s-T_1 \geq \epsilon$ is satisfied along the trajectories of the optimisation dynamics, the second term of this expression satisfies
\begin{equation}\label{projectedHessianComputations3}
	\begin{split}
		-\frac12 YA^{-1}L_m \operatorname{diag}(T^s-T_1)^{-2}&YA^{-1} \\
		&\geq
		-\frac{1}{2\epsilon^2}YA^{-1}L_mYA^{-1}.
	\end{split}
\end{equation}
Substitution of \eqref{projectedHessianComputations2} and \eqref{projectedHessianComputations3} into \eqref{projectedHessianComputations1} reveals that the projected Hessian in \eqref{projectedHessianComputations1} will be positive definite provided that the objective function gains $L_1$, $L_2$, and $L_m$ satisfy
\begin{equation}\label{eg:heating:sufficientCondition}
	\begin{split}
		L_1
		&+
		E_{12}(E_{22} + R_2^{-1})^{-1}L_2(E_{22} + R_2^{-1})^{-1}E_{12}^\top \\
		&-
		\frac{1}{2\epsilon^2}YA^{-1}L_mA^{-1}Y
		>
		0.
	\end{split}
\end{equation}
Consequently, any solution to the optimisation dynamics will be a locally optimal solution.

\subsubsection{Numerical example}
The optimisation problem related to energy-efficient buildings was solved using the parameters
\small
\begin{multicols}{3}
	\noindent
	\begin{equation*}
		\begin{split}
			n_1 &= 1 \\
			T_1^\star &= 23 \\
			C_1 &= 1 \\
			E_{12} &= -0.5 \\
			R_2 &= 2 \\
			A &= 1
		\end{split}
	\end{equation*}
	\begin{equation*}
		\begin{split}
			n_2 &= 1 \\
			T_2^\star &= 23 \\
			C_2 &= 1 \\
			E_{22} &= 0.5 \\
			q &= 0 \\
			\bar m &= 3.
		\end{split}
	\end{equation*}
	\begin{equation*}
		\begin{split}
			T^s &= 30 \\
			T^a &= 10 \\
			E_{11} &= 0.5 \\
			R_1 &= 2 \\
			\epsilon &= 1 \\
		\end{split}
	\end{equation*}
\end{multicols}
\normalsize
\noindent
The objective function gains were chosen as $L_1 = 1$, $L_2 = 1$, and $L_m = 0.1$, which satisfy the sufficient condition for local optimality \eqref{eg:heating:sufficientCondition}. {The optimisation dynamics were initialised at $T_1(0) = 23$, $T_2(0) = 23$, $m(0) = 1$ and two sets of tuning gains $\kappa_1 = 1$, $\kappa_2 = 1$, $\delta T = 0.1$, and $\kappa_1 = 3$, $\kappa_2 = 3$, $\delta T = 0.1$, were used.} 
The resulting trajectory of the optimisation dynamics is shown in Figure \ref{HeatingEg}. The dynamics converge to the point $(T_1, T_2, m) = (25.45, 17.73, 2.55)$. Considering Figure~\ref{HeatingEg}, the first plot shows the evolution of the states along the solution of the optimisation dynamics. The second plot shows the satisfaction of the equality constraints. {The solid lines show the results for the tuning gains $\kappa_1 = 1$, $\kappa_2 = 1$, $\delta T = 0.1$, whereas the dashed lines show the plot for the tuning gains $\kappa_1 = 3$, $\kappa_2 = 3$, $\delta T = 0.1$.} As can be seen in the plot, the constraints are not satisfied at the initial conditions but converge to the origin. This convergence ensures that the solution is indeed an equilibrium of the thermal dynamics~\eqref{heatingDynamics}. Moreover, increasing the tuning gains has increased the rate of convergence to the solution of the optimisation problem as discussed in Remark~\ref{remark:tuning}.

\begin{figure}[ht!]
	\centering{}
	\includegraphics[width=0.90\columnwidth]{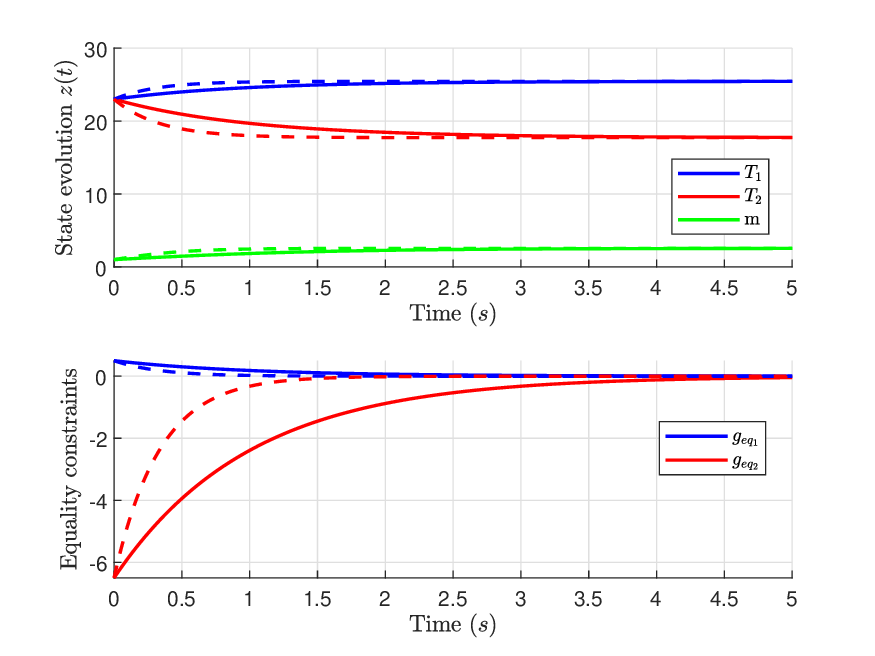}
	\caption{The trajectory of the optimisation dynamics seeking a solution to the optimisation problem related to energy-efficient buildings. The first plot shows the evolution of the optimisation dynamics, whereas the second plot shows the convergence of the equality constraints to the origin. {The solid line indicate the results for tuning gains $\kappa_1 = 1$, $\kappa_2 = 1$, $\delta T = 0.1$, whereas the dashed lines indicate the results for tuning gains $\kappa_1 = 3$, $\kappa_2 = 3$, $\delta T = 0.1$.}}
	\label{HeatingEg}
\end{figure}
\section{Conclusion}
\label{sec:conclusion}
Motivated by the need of a convergent algorithm for solving nonconvex optimisation problems, this paper presents novel switching dynamics that converges to the KKT point of a class of nonconvex optimisation problems subject to both equality and inequality constraints. {Interesting future research avenues of the work include:
\begin{itemize}
	\item The primary limitation of this approach is Assumption \ref{Assumption1}. Future work will investigate how this limitation can be relaxed to broaden the class of applicable problems.
	\item This paper did not address the rate of convergence to a solution. It is expected that under some restriction on the class of admissible objective functions, the rates of convergence (e. g. exponential) could be established.
	\item Continuous-time optimisation can be incorporated into a control scheme for online optimisation, similar to \cite{steginkUnifyingEnergyBasedApproach2017}. In future work, we will investigate how this scheme can be incorporated into control structures. We will also study how to shape the structural properties of the algorithm to allow distributed optimisation and control problems.
\end{itemize}}

\section{Appendix}
\subsection{Execution time comparison with Discrete-Time  Algorithms}
In this paper, we focused on the convergence properties of the proposed optimisation dynamics represented by an ODE with respect to its independent variable $t$. On the other hand, in some applications, it is useful to consider the computational time required to reach a solution (referred to as runtime or execution time). Motivated by this, in this section, we perform a numerical comparison (in terms of execution time) of the proposed optimisation dynamics with three discrete-time methods available in MATLAB via the \texttt{fmincon} function. For this comparison, we considered the real-world building heating optimisation problem presented in Section~\ref{sec:examples:heatSharing}.

To compare the solutions of the optimisation dynamics with discrete-time algorithms, we approximated the solutions using Matlab's \texttt{ode23s} solver with a relative error tolerance of $1\times10^{-6}$ and active constraint tolerance of $1\times10^{-10}$. Note that this choice is independent of the presented theory. A comparison of the trajectories generated by each of the methods is given in Figure~\ref{comparisonFigure}. We also provide in Table~\ref{tab:times}  the total execution time of each of the methods.
\begin{table}[t]
  \caption{Real-time execution of optimization methods}
  \label{tab:times}
  \centering
  \setlength{\tabcolsep}{4pt}           
  \renewcommand{\arraystretch}{1.1}     
  \begin{tabularx}{\columnwidth}{@{}>{\raggedright\arraybackslash}X r@{}}
    \toprule
    \textbf{Algorithm} & \textbf{Time [s]} \\
    \midrule
    Optimisation dynamics (ode23s, relative tolerance \(1\times10^{-6}\)) & 0.8740 \\
    Interior-point                                      & 0.3461 \\
    SQP                                                 & 0.1424 \\
    Active-set                                            & 0.1270 \\
    \bottomrule
  \end{tabularx}
\end{table}

While the proposed method is slower in real-time execution compared to the discrete-time methods, this is not surprising since ODE solvers are generally computationally more expensive. Note that for the example considered in Section~\ref{sec:examples:heatSharing}, the time horizon  for the independent variable (see Fig.~4) is 5~seconds, and the execution time is still significantly faster than real-time. This demonstrates that  the algorithm is suitable for real-time applications such as those  in \cite{steginkUnifyingEnergyBasedApproach2017}. Nevertheless, computational times will depend on the specific problem setup and the choice of ODE solver. 

\begin{figure}[ht!]
	\centering{}
	\includegraphics[width=1\columnwidth]{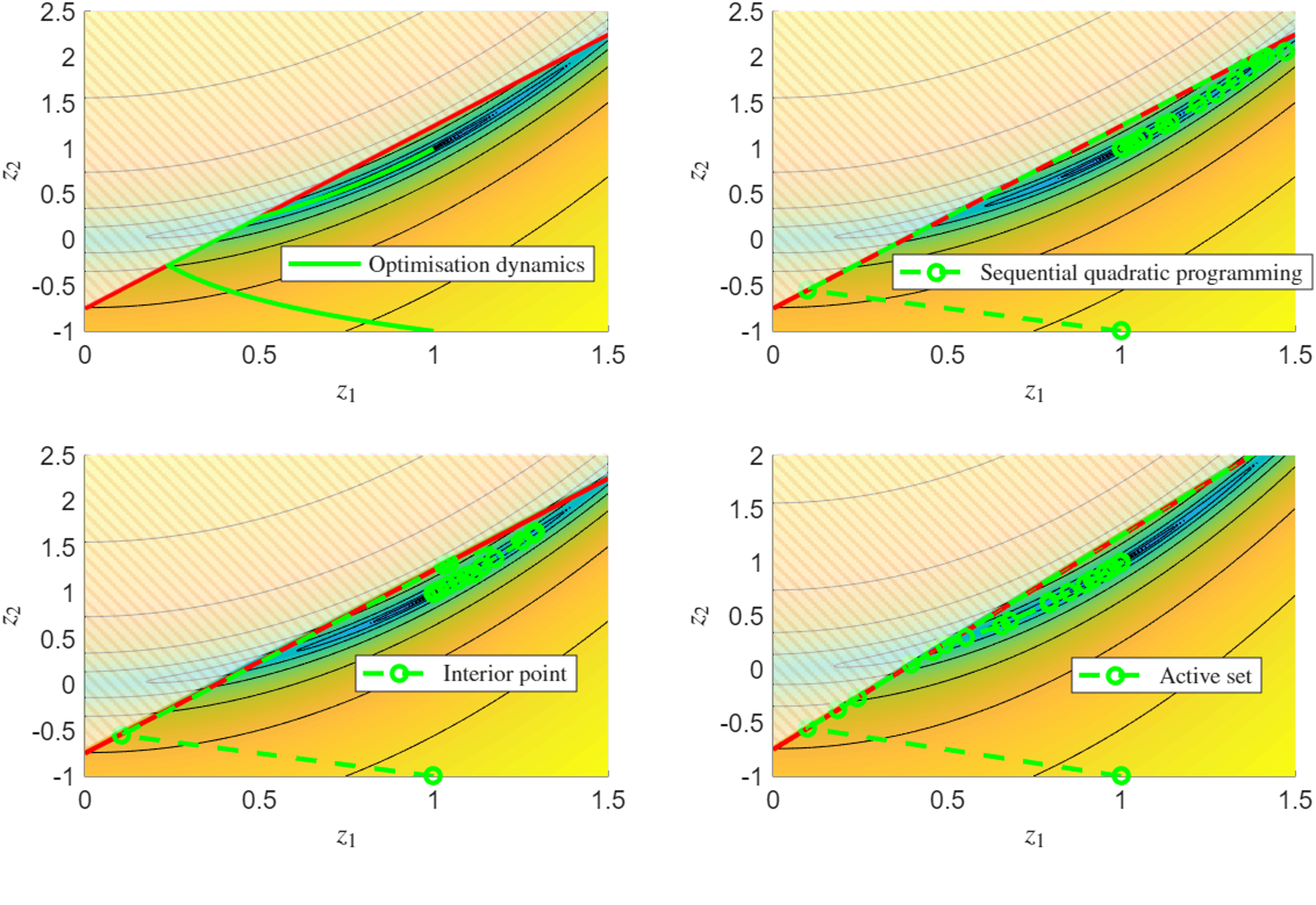}
	\caption{The trajectories generated by the proposed optimisation dynamics and three different discrete-time optimisation algorithms. For each of the SQP, Interior point and Active set methods, the solutions are defined discretely and the steps of the solver are indicated by the $\texttt{o}$ symbols. The optimsation dynamics defines a full trajectory, which has been approximated by Matlab's \texttt{ODE23s} solver and plotted as a solid line in the above figure.}
	\label{comparisonFigure}
\end{figure}

\bibliographystyle{IEEEtran}
\bibliography{library_new}

\vspace{-1em}
\begin{IEEEbiography}[{\includegraphics[width=1in,height=1.25in,clip,keepaspectratio]{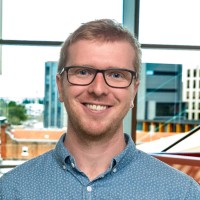}}]{Joel Ferguson} received his B.Eng. (Hons) in Mechatronic Engineering from the University of Newcastle in 2013, being awarded the Dean’s medal and University medal. He received his Ph.D. in Electrical Engineering from the same institution in 2018 where he went on to hold a lecturing position in mechatronics from 2019 - 2024. In 2024 he joined the Maynooth International Engineering College as an assistant professor in electronic engineering. His research interests include nonlinear control, port-Hamiltonian systems and robotics.
\end{IEEEbiography}

\begin{IEEEbiography}[{\includegraphics[width=1in,height=1.25in,clip,keepaspectratio]{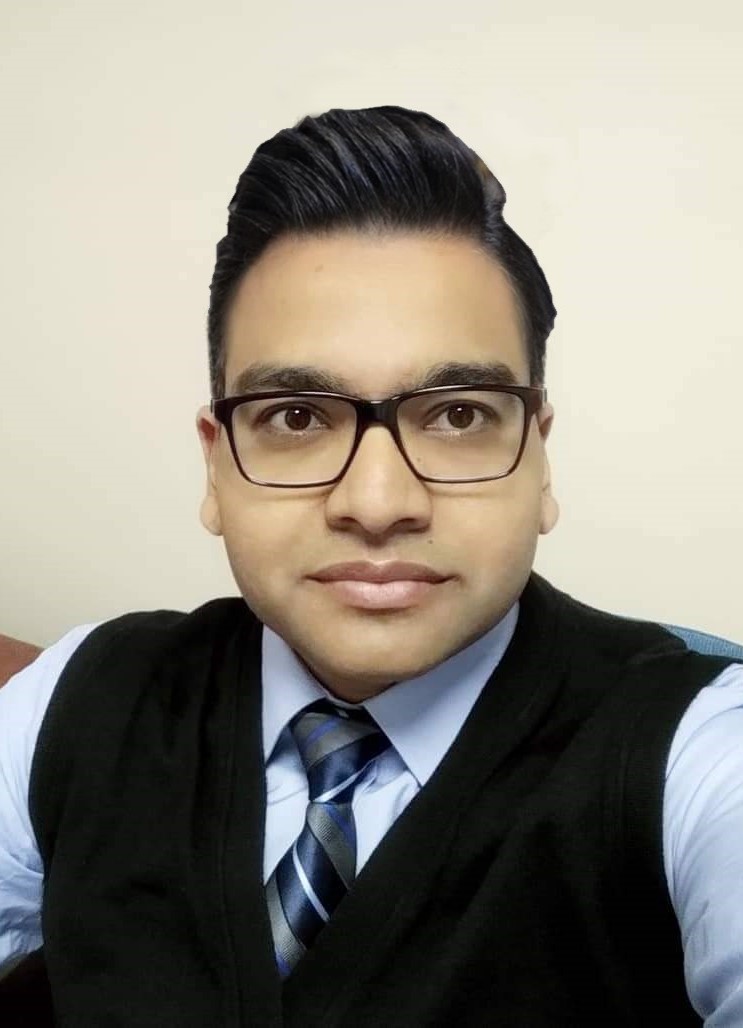}}]{Saeed Ahmed} is an Assistant Professor of Systems and Control at the University of Groningen, where he is affiliated with the Engineering and Technology Institute Groningen and the Jan C. Willems Center for Systems and Control. Prior to joining this position, he held postdoctoral positions at the University of Groningen with Jacquelien Scherpen and at the Technical University of Kaiserslautern (now RPTU), Germany. He completed his Ph.D. at Bilkent University, Turkey. During his Ph.D., he was a visiting scholar at CentraleSupélec, France. His Ph.D. was supervised by  Frederic Mazenc and Hitay Ozbay. He also collaborated with Micheal Malisoff during his Ph.D. His research interests span various topics in systems and control engineering. From a theoretical point of view, he is interested in stability and control, online optimisation, observer design, nonlinear and hybrid (switched and impulsive) systems, dissipativity and passivity analysis, robust control, LPV systems, and time-delay systems. From an application point of view, he is interested in designing intelligent control algorithms for autonomous vehicles and district heating systems. He won several awards including the best presentation award in the Control/Robotics/Communications/Network category at the IEEE Graduate Research Conference 2018 held at Bilkent University, Turkey, and the Outstanding Reviewer Award by the European Journal of Control. He is an associate editor of Systems and Control Letters and IEEE CSS Technology Conference Editorial Board. He is a member of the IFAC Technical Committees on Non-linear Control Systems, Networked Systems, and Distributed Parameter Systems. 
\end{IEEEbiography}

\begin{IEEEbiography}[{\includegraphics[width=1in,height=1.25in,clip,keepaspectratio]{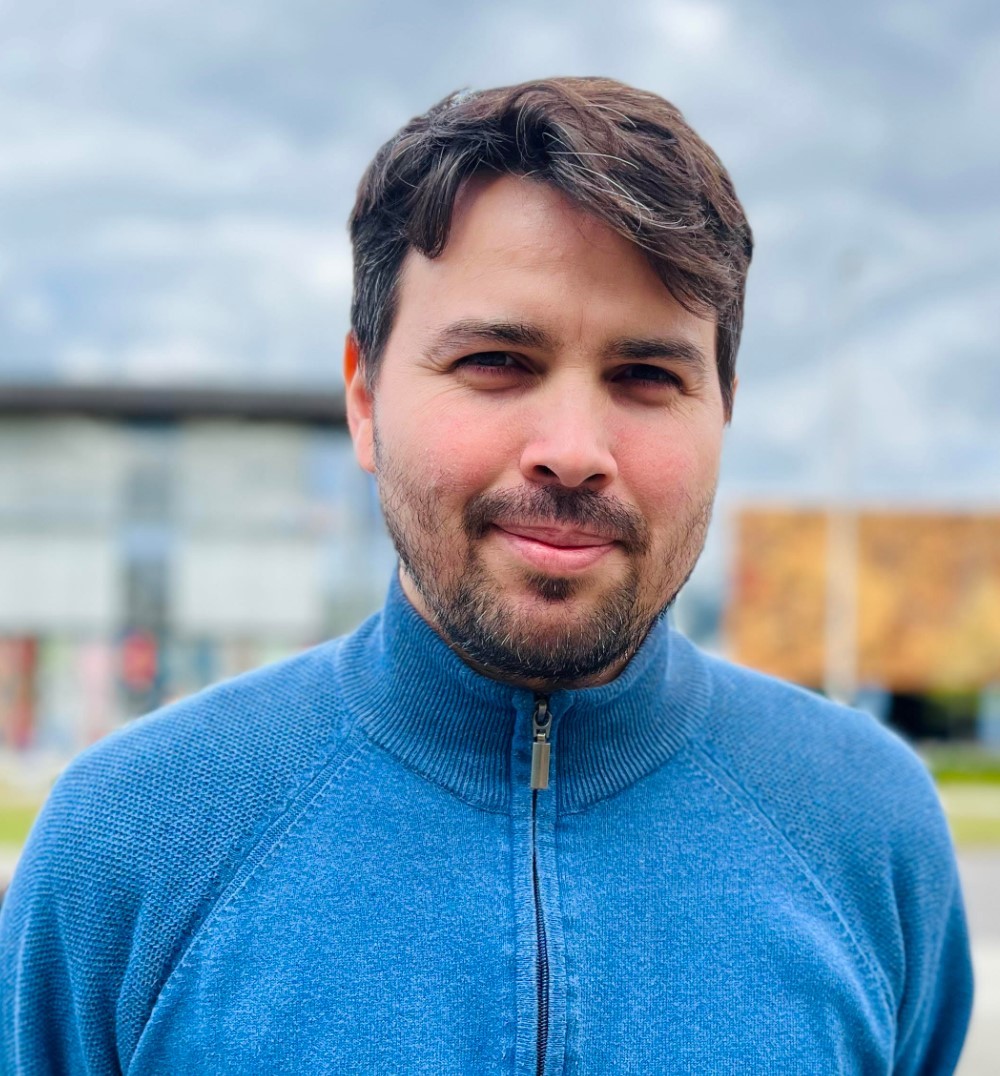}}]{Juan E. Machado}  received the B.Sc. degree in electromechanics from the Technological Institute of
La Paz (ITLP), La Paz, Mexico, in 2012, the M.Sc. degree in applied mathematics from the Center for Research in Mathematics (CIMAT), Guanajuato, Mexico, in 2015, and defended his Ph.D. thesis in automatic control from Paris-Saclay University, Gif-sur-Yvette, France, in 2019. From January 2020 to June 2023, he was a Post-Doctoral Researcher at the Faculty of Science
and Engineering, University of Groningen, Groningen, The Netherlands. Since July 2023, he has been with the Chair of Control Systems and Network Control Technology, Faculty of Mechanical Engineering, Electrical and Energy Systems, Brandenburg University of Technology (BTU), Cottbus, Germany, where he is currently a Junior Research Group Leader of the Young Investigator Group (YIG) “Distributed Control and Operation of Integrated Energy Systems,” which is part of the Project “Control Systems and Cyber Security Lab (COSYS Lab)” of the Energy Innovation Center (EIZ). His research interests are within the analysis and control of nonlinear and networked systems, with a strong emphasis on electric and heat grids
\end{IEEEbiography}

\begin{IEEEbiography}[{\includegraphics[width=1in,height=1.25in,clip,keepaspectratio]{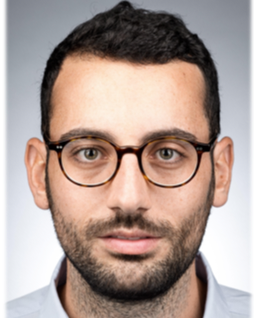}}]{Michele Cucuzzella} (Senior Member, IEEE)  received the
M.Sc. degree (Hons.) in electrical engineering and
the Ph.D. degree in systems and control from the
University of Pavia, Pavia, Italy, in 2014 and 2018,
respectively.
From 2017 to 2020, he held a postdoctoral position
at the University of Groningen (UG), Groningen,
The Netherlands. He then joined the University of
Pavia as an Assistant Professor. In 2024, he moved to
UG as an Associate Professor at the Engineering and
Technology institute Groningen, Faculty of Science
and Engineering. He is also a Visiting Associate Professor at Hiroshima
University, Higashihiroshima, Japan. He has coauthored the book Advanced
and optimisation Based Sliding Mode Control: Theory and Applications
(SIAM, 2019). His research activities are mainly in the area of nonlinear
control with application to the energy domain and smart complex systems.
Dr. Cucuzzella is a member of the EUCA Conference Editorial Board
and the IEEE CSS Technology Conferences Editorial Board. He received
the Certificate of Outstanding Service as a Reviewer of the IEEE CONTROL
SYSTEMS LETTERS 2019. He also received the 2020 IEEE TRANSACTIONS
ON CONTROL SYSTEMS TECHNOLOGY Outstanding Paper Award, the IEEE
Italy Section Award for the best Ph.D. thesis on new technological challenges
in energy and industry, and the SIDRA Award for the best Ph.D. thesis in
the field of systems and control engineering. He was also the finalist for
the EECI Award for the best Ph.D. thesis in Europe in the field of control
for complex and heterogeneous systems and the IEEE-CSS Italy Best Young
Paper Award. He has been serving as an Associate Editor for the European
Journal of Control since 2022.

\end{IEEEbiography}

\begin{IEEEbiography}[{\includegraphics[width=1in,height=1.25in,clip,keepaspectratio]{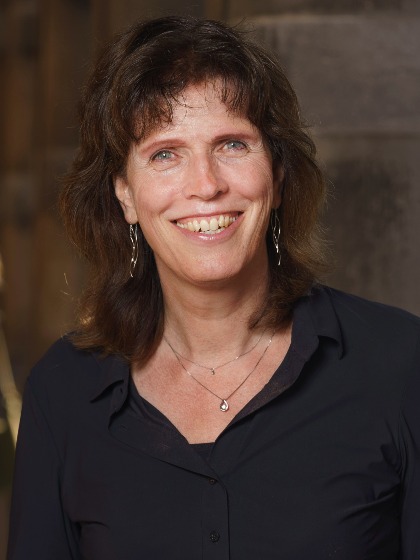}}]{Jacquelien M. A. Scherpen}  (Fellow, IEEE) received
the M.Sc. and Ph.D. degrees from the University
of Twente, Enschede, The Netherlands, in 1990 and
1994, respectively.
She then joined the Delft University of Technology, Delft, The Netherlands, and in 2006, she
moved to the Engineering and Technology Institute
Groningen (ENTEG), Faculty of Science and Engineering, University of Groningen (UG), Groningen,
The Netherlands, where she was the Scientific Director of ENTEG and the Director of engineering. She
is currently the Rector Magnificus of UG. Furthermore, she has been the
Captain of Science of the Dutch Top Sector High Tech Systems and Materials
(HTSM). She has held various visiting research positions at the University
of Tokyo, Tokyo, Japan; Kyoto University, Kyoto, Japan; Old Dominion
University, Norfolk, VA, USA; the Université de Compiègne, Compiègne,
France; and Supélec, Gif-sur-Yvette, France. Her current research interests
include model reduction methods for networks, nonlinear model reduction
methods, nonlinear control methods, modeling and control of physical systems
with applications to electrical circuits, electromechanical systems, mechanical
systems, smart energy networks, and distributed optimal control for smart
grids.
Dr. Scherpen received the 2017–2020 Automatica Best Paper Prize. In 2019,
she received the Royal Distinction as Knight in the Order of the Netherlands
Lion. In 2023, she was awarded the Prince Friso Prize for Engineer of the
Year in The Netherlands. She has been active at the International Federation
of Automatic Control (IFAC) and the IEEE Control Systems Society. She
was the President of the European Control Association (EUCA) and has
chaired the SIAM Activity Group on Control and Systems Theory. She has
been on the Editorial Board of a few international journals among which
are IEEE TRANSACTIONS ON AUTOMATIC CONTROL and the International
Journal of Robust and Nonlinear Control.
\end{IEEEbiography}

\end{document}